\documentclass[12pt]{amsart}
\usepackage{amsmath,amsthm}
\usepackage{amssymb}
\usepackage[all]{xypic}
\SelectTips{cm}{12}
\UseTips{}
\usepackage{euscript}
\begin{document}
%
%
\theoremstyle{plain}
\swapnumbers
    \newtheorem{thm}{Theorem}[section]
    \newtheorem{prop}[thm]{Proposition}
    \newtheorem{lemma}[thm]{Lemma}
    \newtheorem{cor}[thm]{Corollary}
    \newtheorem{subsec}[thm]{}
    \newtheorem*{thma}{Theorem A}
    \newtheorem*{thmb}{Theorem B}
    \newtheorem*{thmc}{Theorem C}
\theoremstyle{definition}
    \newtheorem{defn}[thm]{Definition}
    \newtheorem{example}[thm]{Example}
    \newtheorem{notn}[thm]{Notation}
\theoremstyle{remark}
    \newtheorem{remark}[thm]{Remark}
    \newtheorem{ack}[thm]{Acknowledgements}
%
%
\newenvironment{myeq}[1][]
{\stepcounter{thm}\begin{equation}\tag{\thethm}{#1}}
{\end{equation}}
\newcommand{\mydiag}[2][]{\myeq[#1]\xymatrix{#2}}
\newcommand{\mydiagram}[2][]
{\stepcounter{thm}\begin{equation}
     \tag{\thethm}{#1}\vcenter{\xymatrix{#2}}\end{equation}}
%
\newenvironment{mysubsection}[2][]
{\begin{subsec}\begin{upshape}\begin{bfseries}{#2.}
\end{bfseries}{#1}}
{\end{upshape}\end{subsec}}
\newenvironment{mysubsect}[2][]
{\begin{subsec}\begin{upshape}\begin{bfseries}{#2\vsn.}
\end{bfseries}{#1}}
{\end{upshape}\end{subsec}}
\newcommand{\sect}{\setcounter{thm}{0}\section}
\newcommand{\wh}{\ -- \ }
\newcommand{\wwh}{-- \ }
\newcommand{\w}[2][ ]{\ \ensuremath{#2}{#1}\ }
\newcommand{\ww}[1]{\ \ensuremath{#1}}
\newcommand{\www}[2][ ]{\ensuremath{#2}{#1}\ }
\newcommand{\wb}[2][ ]{\ (\ensuremath{#2}){#1}\ }
\newcommand{\wwb}[1]{\ (\ensuremath{#1})-}
\newcommand{\wref}[2][ ]{\ \eqref{#2}{#1}\ }
\newcommand{\wwref}[1]{\ \eqref{#1}}
%
%
\newcommand{\xra}[1]{\xrightarrow{#1}}
\newcommand{\xla}[1]{\xleftarrow{#1}}
\newcommand{\larr}{\leftarrow}
\newcommand{\xsim}{\xrightarrow{\sim}}
\newcommand{\hra}{\hookrightarrow}
\newcommand{\llra}{\longrightarrow}
\newcommand{\epic}{\to\hspace{-5 mm}\to}
\newcommand{\adj}[2]{\substack{{#1}\\ \rightleftharpoons \\ {#2}}}
\newcommand{\ccsub}[1]{\circ_{#1}}
\newcommand{\DEF}{:=}
\newcommand{\EQUIV}{\Leftrightarrow}
\newcommand{\hsp}{\hspace{10 mm}}
\newcommand{\hs}{\hspace*{6 mm}}
\newcommand{\hsm}{\hspace*{3 mm}}
\newcommand{\hso}{\hspace*{2 mm}}
\newcommand{\vsm}{\vspace{3 mm}}
\newcommand{\vsn}{\vspace{1 mm}}
\newcommand{\vs}{\vspace{4 mm}}
\newcommand{\rest}[1]{\lvert_{#1}}
\newcommand{\lra}[1]{\langle{#1}\rangle}
\newcommand{\lin}[1]{\{{#1}\}}
\newcommand{\llrr}[1]{\langle\!\langle{#1}\rangle\!\rangle}
\newcommand{\sss}{\hspace*{1 mm}\sp{s}}
\newcommand{\bstar}{\mbox{\large $\star$}}
\newcommand{\q}[1]{^{({#1})}}
\newcommand{\li}[1]{_{({#1})}}
%
%
\newcommand{\A}{{\EuScript A}}
\newcommand{\Aj}[1]{A\li{#1}}
\newcommand{\bA}{\bar{A}}
\newcommand{\bAj}[1]{\bA\li{#1}}
\newcommand{\B}{{\EuScript B}}
\newcommand{\C}{{\mathcal C}}
\newcommand{\sC}{\sss\C}
\newcommand{\CsA}{\C\sb{A}}
\newcommand{\hCsA}{\widehat{\C}\sb{A}}
\newcommand{\CasA}[1]{\C\sp{#1}\sb{A}}
\newcommand{\CesA}{\CasA{\lambda}}
\newcommand{\hCasA}[1]{\widehat{\C}\sp{#1}\sb{A}}
\newcommand{\hCesA}{\hCasA{\lambda}}
\newcommand{\hCesX}{\widehat{\C}\sp{e}\sb{X}}
\newcommand{\CX}{\C\sb{X}}
\newcommand{\hCesY}{\widehat{\C}\sp{e}\sb{Y}}
\newcommand{\CY}{\C\sb{Y}}
\newcommand{\hCesZ}{\widehat{\C}\sp{e}\sb{Z}}
\newcommand{\CZ}{\C\sb{Z}}
\newcommand{\tCasA}[1]{\widetilde{\C}\sp{#1}\sb{A}}
\newcommand{\tCesA}{\tCaSA{\lambda}}
\newcommand{\CuA}{\C\sp{A}}
\newcommand{\D}{{\mathcal D}}
\newcommand{\E}{{\EuScript E}}
\newcommand{\F}{{\EuScript F}}
\newcommand{\cF}{{\mathcal F}}
\newcommand{\Fp}{\F'}
\newcommand{\FA}{F_{A}}
\newcommand{\G}{{\mathcal G}}
\newcommand{\II}{{\mathcal I}}
\newcommand{\K}{{\mathcal K}}
\newcommand{\tK}{\tilde{K}}
\newcommand{\LL}{{\mathcal L}}
\newcommand{\tLA}{\tilde{L}_{A}}
\newcommand{\N}{{\EuScript N}}
\newcommand{\OO}{{\EuScript O}}
\newcommand{\PP}{{\mathcal P}}
\newcommand{\Ss}{{\mathcal S}}
\newcommand{\Sa}{\Ss\sb{\ast}}
\newcommand{\Sk}{\Ss\sb{\ast}\sp{\Kan}}
\newcommand{\TT}{{\mathcal T}}
\newcommand{\Ta}{\TT\sb{\ast}}
\newcommand{\V}{{\mathcal V}}
%
%
\newcommand{\hy}[2]{{#1}\text{-}{#2}}
%
%
\newcommand{\Alg}[1]{{#1}\text{-}{\EuScript Alg}}
\newcommand{\Ab}{{\EuScript Ab}}
\newcommand{\Abgp}{{\Ab\Grp}}
\newcommand{\Cat}{{\EuScript Cat}}
\newcommand{\Grp}{{\EuScript Gp}}
\newcommand{\Gpd}{{\EuScript Gpd}}
\newcommand{\Set}{{\EuScript Set}}
\newcommand{\Sets}{\Set_{\ast}}
\newcommand{\VC}{\hy{\V}{\Cat}}
%
%
\newcommand{\GO}{(\Gpd,\OO)}
\newcommand{\GOC}{\hy{\GO}{\Cat}}
\newcommand{\GOp}{(\Gpd,\Op)}
\newcommand{\GOpC}{\hy{\GOp}{\Cat}}
\newcommand{\OC}{\hy{\OO}{\Cat}}
\newcommand{\SO}{(\Ss,\OO)}
\newcommand{\SaO}{(\Sa,\OO)}
\newcommand{\SOC}{\hy{\SO}{\Cat}}
\newcommand{\SaOC}{\hy{\SaO}{\Cat}}
\newcommand{\SOp}{(\Ss,\Op)}
\newcommand{\SOpC}{\hy{\SOp}{\Cat}}
\newcommand{\SaOp}{(\Sa,\Op)}
\newcommand{\SaOpC}{\hy{\SaOp}{\Cat}}
\newcommand{\VO}{(\V,\OO)}
\newcommand{\VOC}{\hy{\VO}{\Cat}}
%
%
\newcommand{\FFp}{{\mathbb F}\sb{p}}
\newcommand{\bQ}{{\mathbb Q}}
\newcommand{\bR}{{\mathbb R}}
\newcommand{\bZ}{{\mathbb Z}}
\newcommand{\Fn}{\{K(\FFp,n)\}_{n=1}^{\infty}}
%
%
\newcommand{\ab}{\sb{\operatorname{ab}}}
\newcommand{\AQ}{\sb{\operatorname{AQ}}}
\newcommand{\Coeq}{\operatorname{Coeq}}
\newcommand{\CW}{\operatorname{CW}}
\newcommand{\CWA}[1]{\CW\hspace{-1.5mm}\sb{A}{#1}}
\newcommand{\colim}{\operatorname{colim}}
\newcommand{\seg}[1]{\operatorname{Seg}[{#1}]}
\newcommand{\eval}{\operatorname{eval}}
\newcommand{\hocolim}{\operatorname{hocolim}}
\newcommand{\ho}{\operatorname{ho}}
\newcommand{\holim}{\operatorname{holim}}
\newcommand{\Hom}{\operatorname{Hom}}
\newcommand{\uHom}{\underline{\Hom}}
\newcommand{\Id}{\operatorname{Id}}
\newcommand{\inc}{\operatorname{inc}}
\newcommand{\Kan}{\operatorname{Kan}}
\newcommand{\Ker}{\operatorname{Ker}}
\newcommand{\Obj}{\operatorname{Obj}\,}
\newcommand{\op}{\sp{\operatorname{op}}}
\newcommand{\St}{\operatorname{St}}
%
%
\newcommand{\map}{\operatorname{map}}
\newcommand{\mapp}{\map\!}
\newcommand{\mapa}{\map_{\ast}}
\newcommand{\mape}[2]{\map\sp{#1}\sb{#2}}
%
%
\newcommand{\bDelta}{\mathbf{\Delta}}
\newcommand{\var}{\varepsilon}
\newcommand{\hvar}{\widehat{\var}}
\newcommand{\tvar}{\widetilde{\var}}
\newcommand{\hee}[1]{\widehat{h}\sp{\lambda}\sb{#1}}
\newcommand{\tet}[1]{\widetilde{h}\sb{#1}}
\newcommand{\bS}[1]{\mathbf{S}^{#1}}
\newcommand{\qe}[1]{e^{(#1)}}
\newcommand{\fff}{\mathfrak{f}}
\newcommand{\hff}{\hat{\fff}}
\newcommand{\qf}[1]{\fff\q{#1}}
\newcommand{\tqf}[1]{\tilde{\fff}\q{#1}}
\newcommand{\wqf}[1]{\widetilde{\qf{#1}}}
\newcommand{\qi}[1]{i^{(#1)}}
\newcommand{\tqi}[1]{\widetilde{\imath}\q{#1}}
\newcommand{\hqi}[1]{{\widehat{\imath}}\q{#1}}
\newcommand{\qp}[1]{p^{(#1)}}
\newcommand{\tqp}[1]{\tilde{p}\q{#1}}
\newcommand{\hr}{\widehat{\rho}}
\newcommand{\qs}[1]{s\sp{(#1)}}
\newcommand{\tqs}[1]{\tilde{s}\q{#1}}
\newcommand{\qY}[1]{Y^{(#1)}}
\newcommand{\hqY}[1]{\widehat{Y}\q{#1}}
\newcommand{\tqY}[1]{\widetilde{Y}\q{#1}}
\newcommand{\qZ}[1]{Z\q{#1}}
\newcommand{\qF}[1]{\F\q{#1}}
%
%
\newcommand{\Fd}{\cF_{\bullet}}
\newcommand{\Md}{M_{\bullet}}
\newcommand{\Vd}{V_{\bullet}}
\newcommand{\Wd}{W_{\bullet}}
\newcommand{\Xd}{X_{\bullet}}
%
%
\newcommand{\TA}{T_{A}}
\newcommand{\MA}{M_{A}}
\newcommand{\ma}[1][ ]{mapping algebra{#1}}
\newcommand{\Ama}[1][ ]{$A$-mapping algebra{#1}}
\newcommand{\Fma}{$\F$-\ma}
\newcommand{\qFma}[1]{\ww{\qF{#1}}-mapping algebra}
\newcommand{\Bma}[1][ ]{$\B$-mapping algebra{#1}}
\newcommand{\Tal}{$\bT$-algebra}
\newcommand{\bT}{\mathbf{\Theta}}
\newcommand{\TsA}{\bT\sb{A}}
\newcommand{\Tss}[2]{\bT\sp{#1}\sb{#2}}
\newcommand{\TssA}{\Tss{\St}{A}}
\newcommand{\hTsA}{\hat{\bT}\sb{A}}
\newcommand{\ThesA}{\bT\sp{e}\sb{A}}
\newcommand{\TisA}{\Tss{\II}{A}}
\newcommand{\TesA}{\qF{\lambda}}
\newcommand{\TsF}{\bT\sb{\F}}
\newcommand{\TuFp}{\bT\sp{\Fp}}
\newcommand{\M}{{\EuScript Map}}
\newcommand{\MsA}{\M\sb{A}}
\newcommand{\MsAI}{\M\sb{A}\sp{\II}}
\newcommand{\MsAs}{\M\sb{A}\sp{\St}}
\newcommand{\MsAd}{\MsA\sp{\delta}}
\newcommand{\MsAsd}{\MsA\sp{\St,\delta}}
\newcommand{\MsAa}[1]{\widehat{\M\sb{A}\sp{#1}}}
\newcommand{\MsAe}{\MsAa{\lambda}}
\newcommand{\MsF}{\M\sb{\F}}
\newcommand{\MsqF}[1]{\M\sp{#1}}
\newcommand{\MuFp}{\M\sp{\Fp}}
\newcommand{\Ld}{L^{\delta}}
\newcommand{\LAd}{L_{A}^{\delta}}
\newcommand{\wLA}[1]{\widetilde{L_{A}{#1}}}
%
%
\newcommand{\fD}[1]{\mathfrak{D}\sp{e}\!{#1}}
\newcommand{\fL}{\mathfrak{L}}
\newcommand{\fLA}{\fL\sb{A}}
\newcommand{\fM}{\mathfrak{M}}
\newcommand{\fMA}{\fM\sb{A}}
\newcommand{\fMAd}{\fM\sb{A}\sp{\delta}}
\newcommand{\fMAs}{\fM\sb{A}\sp{\St}}
\newcommand{\fMAsd}{\fM\sb{A}\sp{\St,\delta}}
\newcommand{\fMAa}[1]{\widehat{\fM}\sb{A}\sp{#1}}
\newcommand{\fMAe}{\fMAa{\lambda}}
\newcommand{\fMF}{\fM\sb{\F}}
\newcommand{\fMqF}[1]{\fM\sp{#1}\sb{A}}
\newcommand{\fuMFp}{\fM\sp{\Fp}}
\newcommand{\fT}{\mathfrak{T}}
\newcommand{\fTA}{\fT_{A}}
\newcommand{\fTAd}{\fTA\sp{\delta}}
\newcommand{\fTAsd}{\fTA\sp{\St,\delta}}
\newcommand{\fTAe}{\hat{\fT}\sb{A}\sp{\lambda}}
\newcommand{\fTqF}[1]{\widetilde{\fT}\sp{#1}}
\newcommand{\fTF}{\fT_{\F}}
\newcommand{\fV}{\mathfrak{V}}
\newcommand{\fVd}{\fV_{\bullet}}
\newcommand{\fW}{\mathfrak{W}}
\newcommand{\fWd}{\fW_{\bullet}}
\newcommand{\fX}{\mathfrak{X}}
\newcommand{\tX}{\tilde{\fX}}
\newcommand{\hfX}{\hat{\fX}}
\newcommand{\tfX}{\tilde{\fX}}
\newcommand{\fY}{\mathfrak{Y}}
\newcommand{\hfY}{\hat{\fY}}
\newcommand{\tfY}{\tilde{\fY}}
\newcommand{\fZ}{\mathfrak{Z}}
%
%
\title{Recognizing mapping spaces}
%
%
\author [B.~Badzioch]{Bernard Badzioch}
\address{Dept.\ of Mathematics\\ University at Buffalo, SUNY\\ 
Buffalo, NY 14260\\ USA }
\author[D.~Blanc]{David Blanc}
\address{Dept.\ of Mathematics\\ University of Haifa\\ 31905 Haifa\\ Israel}
\author[W.~Dorabia{\l}a]{Wojciech Dorabia{\l}a}
\address{Dept.\ of Mathematics\\  Penn State Altoona\\ Altoona,
PA 16601\\ USA}
\date{\today}
\subjclass{Primary: 55Q35; \ secondary: 55N99, 55S20, 18G55}
\keywords{Mapping spaces, mapping algebras, loop spaces, classifying space}

\begin{abstract}
Given a fixed object $A$ in a suitable pointed simplicial model category $\C$, we study the problem of recovering the target $Y$ from the pointed mapping space \w{\mapa(A,Y)} (up to $A$-equivalence). We describe a recognition principle, modelled on the classical ones for loop spaces, but using the more general notion of an \emph{\Ama[.]} It has an associated transfinite procedure for recovering \w{\CWA Y} from \w[,]{\mapa(A,Y)} inspired by Dror-Farjoun's construction of \ww{\CWA{}}-approximations.
\end{abstract}
\maketitle

\setcounter{section}{0}

%
%
\section*{Introduction}
\label{cint}

Given two topological spaces $A$ and $Y$, the set \w{\Hom_{\TT}(A,Y)} of
all continuous maps from $A$ to $Y$ can be endowed with the compact-open
topology, forming the space \w{\uHom(A,Y)} of maps from $A$ to $Y$. Up to homotopy, this may be replaced by the simplicial set \w{\mapp(A,Y)}
(with \w[,]{\mapp(A,Y)_{n}:=\Hom_{\TT}(A\otimes\Delta[n],Y)}
where \w{\Delta[n]} is the standard simplicial $n$-simplex).
Mapping spaces, pointed and unpointed, play a central role in
homotopy theory, and two basic questions arising in their study are:

\begin{enumerate}
\renewcommand{\labelenumi}{(\alph{enumi})~}
\item Given a space $A$, when is $X$ of the form \w{\mapp(A,Y)} for some space
  $Y$ (up to weakly equivalence)?
\item If $X$ satisfies the conditions prescribed in the answer to (a),
  how can we recover $Y$ from \w[?]{X\simeq\mapp(A,Y)}
\end{enumerate}

In this paper we concentrate on \emph{pointed} mapping spaces, since these appear to be more useful in applications. The special case when \w{A=S^{n}} has
  been studied intensively since the 1950s \wh see, e.g.,
  \cite{SugH,SugHS,StaH,MayG,JSmiS,CobbP,BadzA}.
Less has been done on the general problem (but see
\cite{BDoraL,BlaM,CDInteA,FTanHS,IStromH}).

\begin{remark}\label{rrecover}
Note that the idea of recovering $Y$ itself from \w{X:=\mapa(A,Y)} does not really make sense: this is evident even for  \w[,]{A=S^{1}} where
\w{\Omega Y=\mapa(A,Y)} depends only on the base-point component of $Y$. In fact, \w{\Omega Y} is weakly equivalent to \w[,]{\Omega(\CW Y)} where \w{\CW Y} is the usual CW-approximation of $Y$.  More generally, for any \w{A,Y\in\Ta} we have \w[,]{\mapa(A,Y)\simeq\mapa(A,\CWA{Y})} where \w{\CWA{Y}} is the $A$-CW approximation of $Y$ (cf.\ \cite[\S 2]{DroC}).
Thus the best we can hope for is to recover $Y$ from \w{\mapa(A,Y)} \emph{up to $A$-equivalence} (see Definition \ref{daequiv} below).
\end{remark}

\begin{mysubsection}{Summary of results}\label{ssumr}
Our approach is based on the notion of an \emph{\Ama[,]} first introduced in
\cite{BBlaC}): this encodes the action of mapping spaces between various colimits constructed out of $A$ on the corresponding limits of \w[.]{\mapa(A,Y)} It can be thought of as a continuous version of Lawvere's concept of an algebra over a theory $\Theta$, replacing the latter by a simplicially enriched category \w[,]{\TsA} with \w{\fX:\TsA\op\to\Sa} a continuous product-preserving functor.

In fact, one can define the notion of an extended \Ama by adding certain (homotopy) colimits to \w{\TsA} (to obtain a suitable enriched sketch \w[,]{\TsF} in the sense of Ehresmann \wh cf.\ \cite{EhreET}), and requiring $\fX$ to take these colimits to limits in \w{\Sa}
(see Section \ref{cma}).

A judicious choice of the colimits we allow provides enough extra structure on \w{\mapa(A,-)} to be able to recognize the image of this functor. The technical apparatus needed to specify these choices, using the Stover construction of \cite{StoV}, is given in Section \ref{csc},  with the connection to mapping algebras explained in Section \ref{cmona}.

In Section \ref{cicc} we explain how to use this extra structure on $X$ to inductively construct a transfinite sequence of objects \w{\qY{\alpha}} as appropriate homotopy colimits, and show:

\begin{thma}
Any pointed simplicial set $X$ equipped with an extended \Ama structure $\hfX$
is weakly equivalent to \w[,]{\mapa(A,\hocolim_{\alpha<\lambda}\qY{\alpha})}
where \w{(\qY{\alpha})\sb{\alpha<\lambda}} is the sequence of objects  constructed in Section \ref{cicc}.
\end{thma}
\noindent See Theorem \ref{treal} below. This provides a theoretical answer to the first question, too \wh although, as with many recognition principles, it is essentially impossible to verify in practice\vsm .

There is also a homotopy-invariant version of
the second approach, in which we replace \Ama[s] by their lax version
(see \S \ref{slat}), using an appropriate cofibrant replacement of
\w[,]{\TsA} described in Section \ref{chiv}. We then show:

\begin{thmb}
If $X$ is a pointed simplicial set equipped with a lax extended \Ama
structure, and the sequence \w{(\tqY{\alpha})\sb{\alpha<\lambda}}
is constructed in \S \ref{slscc}, then \w[,]{X\simeq\mapa(A,Y)} for \w[.]{Y:=\hocolim_{\alpha<\lambda}\tqY{\alpha}} Moreover, if
\w[,]{X\simeq X'} then \w{X'} also has an induced lax extended \Ama
structure, from which we can recover \w{Y'} ($A$-equivalent to $Y$) with \w[.]{X'\simeq\map(A,Y')}.
\end{thmb}
\noindent See Theorem \ref{threal} below\vsm .

In \S \ref{slsr} we observe that when $A$ is a wedge of (possibly
localized) spheres of various dimensions, we can use a lax version of
Stover's construction (see \cite{StoV}) to recover $Y$ up to homotopy
from \w[.]{X\simeq\mapa(A,Y)} This allows one to bypass the anomalous
behavior of the usual ``homotopy multiplicative'' loop space
recognition and recovery machinery described in \cite{BDoraL}.

Finally, in the Appendix we sketch an alternative ``Dold-Lashof'' approach to recovering the target of a mapping space, in the form of an obstruction theory based on the construction of \cite[\S 2B]{DroC} (see Theorem \ref{tobstr}).
\end{mysubsection}

\begin{mysubsection}{Notation and conventions}\label{snot}
The category of compactly generated Hausdorff spaces (cf.\ \cite{SteCC}, and compare \cite{VogtCC}) is denoted by $\TT$, and that of pointed connected compactly generated spaces by \w[.]{\Ta}
For any category $\C$, \w{s\C:=\C^{\bDelta\op}} is the category of
simplicial objects over $\C$. We abbreviate \w{s\Set} to \w[,]{\Ss}
\w{s\Sets} to \w[,]{\Sa} and \w{s\Grp} to \w[.]{\G} The full subcategory of fibrant objects (Kan complexes) in \w{\Sa} is denoted by \w[.]{\Sk}

If \w{\lra{\V,\otimes}} is a monoidal category, we denote by \w{\VC}
the collection of all (not necessarily small) categories enriched over $\V$
(see \cite[\S 6.2]{BorcH2}). For any set $\OO$, denote by \w{\OC} the
category of all small categories $\D$ with \w[.]{\Obj\D=\OO} A
\ww{\VO}-\emph{category} is a category \w{\D\in\OC} enriched over
$\V$, with $\V$-mapping objects \w[.]{\mape{v}{\D}(-,-)}
The category of all \ww{\VO}-categories will be denoted by \w[.]{\VOC}

The main examples of \w{\lra{\V,\otimes}} we have in mind are
\w[,]{\lra{\Set,\times}} \w[,]{\lra{\Grp,\times}} \w[,]{\lra{\Gpd,\times}}
\w[,]{\lra{\Ss,\times}} and \w[.]{\lra{\Sa,\wedge}}
\end{mysubsection}

\begin{mysubsection}{Assumptions}\label{sass}
Throughout this paper we consider only pointed simplicial model categories $\C$: This means that in addition to the pointed simplicial mapping space
\w{\mapa(X,Y)=\map_{\C}(X,Y)\in\Sa} for any \w[,]{X,Y\in\C} we have two bifunctors \w{-\otimes-:\C\times\Ss\to\C} and \w[,]{(-)^{(-)}:\C\times\Ss\op\to\C} as in \cite[II, \S 1-2]{QuiH}.

We may assume that the $n$-simplices of \w{\mapa(X,Y)} are given by \w[,]{\mapa(X,Y)_{n}:=\Hom_{\C}(X\otimes\Delta[n],Y)}
where \w{\Delta[n]} is the standard simplicial $n$-simplex in $\Ss$.

We assume further that $\C$ has all limits and colimits, is cofibrantly generated and left proper (cf.\ \cite[\S 13.1]{PHirM}, with functorial factorizations (cf.\ \cite[\S 1.1.1]{HovM}), such that all objects are fibrant  (e.g., \w{\C=\Ta}
or $\G$).  We make an additional assumption on the simplicial structure in
\S \ref{reso} below\vsm .

When this does not hold \wh for example, for \w{\C=\Sa} \wwh one can sometimes use an appropriate Dwyer-Kan equivalence \w{R:\C\to \D} to a category $\D$ with all objects fibrant to produce a new enrichment \w{\map'_{\C}(A,Y):=\map_{\D}(RA,RY)} whose properties we can then study by the methods described here.  We shall not pursue this idea further in this paper.
\end{mysubsection}

\begin{ack}
The authors wish to thank Boris Chorny for useful comments on this paper, and the referee for many detailed and pertinent remarks.  The second author was partially supported by Israel Science Foundation grant 74/11.
\end{ack}

%
%
\sect{Mapping algebras}
\label{cma}

Any mapping space \w{X=\mapa(A,Y)\in\Sa} (for fixed $A$) has
additional structure, given by the action of \w{\mapa(A,A)} on $X$,
and so on. This can be useful in recovering $Y$ from $X$, as we know
from the example of \w{A=S^{1}} (see \cite{SugH,StaH}).
In order to codify this extra structure, we shall use the following version of 
the notion of a \emph{\ma[,]} introduced in \cite[\S 8]{BBlaC} 
(see also \cite[\S 2]{BBlaS}).

Recall that for a pointed simplicial model category $\C$ as in \S \ref{sass}, one can define the $n$-fold suspension \w{\Sigma\sp{n} A} of any \w{A\in\C} to be \w{A\otimes\bDelta[n]/A\otimes\partial\bDelta[n]}  (see \cite[I, \S 2]{QuiH} and compare \cite[\S 5.2]{DKStE}). We let \w[.]{\Sigma A:=\Sigma\sp{1}A}

\begin{defn}\label{dmapalg}
Let $A$ be a fixed cofibrant object in $\C$, and let \w{\TsA} be the full sub-simplicial category of $\C$ generated by $A$ under suspensions 
and coproducts of cardinality \www[,]{<\kappa} for a suitably chosen 
cardinal $\kappa$ (see \S \ref{sia} below). More generally, we can let 
$\F$ be any set of cofibrant objects in $\C$, with \w{\TsF} the full sub-simplicial category of $\C$ generated by $\F$ under
suspensions and coproducts of cardinality \www[.]{<\kappa}
A simplicial functor \w[,]{\fX:\TsF\op\to\Sk} written 
\w[,]{B\mapsto\fX\lin{B}} is called an \emph{\Fma} if $\fX$ takes coproducts 
in \w{\TsA} to products and suspensions to loops:
\begin{myeq}\label{eqcoploop}
\fX\lin{\coprod_{i}\,B\sb{i}}~\cong~\prod\sb{i}\,\fX\lin{B\sb{i}}
\hs\text{and}\hs
\fX\lin{\Sigma B}~\cong~\Omega\fX\lin{B}~
\end{myeq}
\noindent for \w[.]{B, B_{i}\in\F} Note that each \w{\fX\lin{B}} is fibrant, by assumption. The category of \Fma[s] will be denoted by \w[.]{\MsF} In particular, when \w[,]{\F=\{A\}} an \Fma is called an \emph{\Ama[,]} and the category of \Ama[s] will be denoted by \w[.]{\MsA}
\end{defn}

\begin{remark}\label{roppcat}
One could rephrase \wref{eqcoploop} more succinctly by requiring that $\fX$ preserve the loops and products of the \emph{opposite} category \w[.]{\TsF\op\subseteq\C\op} However, since the main example we have in mind is \w[,]{\C=\Ta} we thought it would be easier for the reader to keep track of the constructions in the original category $\C$.
\end{remark}

\begin{defn}\label{dreal}
For each \w{Y\in\C} we have a \emph{realizable} \Fma \w[,]{\fMF Y}
defined by setting \w{(\fMF Y)\lin{B}:=\map_{\C}(B,Y)} for each
\w[.]{B\in\TsF} When $\F$ is generated by $A$, we write \w{\fMA Y} for \w[.]{\fMF Y}
\end{defn}

\begin{example}\label{egmapalg}
When \w[,]{\C=\Ta} \w[,]{A=\bS{1}} and \w[,]{\kappa=\omega} we see
that \w{\TsA} is the full sub-simplicial category of \w{\Ta} whose
objects are finite wedges of spheres. In this case an \Ama is an
enriched version of a $\Pi$-algebra (cf.\ \cite[\S 4]{StoV}) and a \emph{realizable} \Ama[,] denoted by \w[,]{\fMA Y} is just a loop space \w{Z=\Omega Y} equipped with an action of all loop spaces of wedges of spheres on products of iterated loops of $Z$.
\end{example}

\begin{defn}\label{daequiv}
If $A$ is a fixed cofibrant object $A$ in a pointed simplicial model category $\C$ as above, we say that a map \w{f:X\to Y} in $\C$ is an $A$-\emph{equivalence} if the induced map \w{f_{\ast}:\mapp_{\C}(A,X)\to\mapp_{\C}(A,Y)} is a weak equivalence of pointed simplicial sets. In particular, the natural map \w{\CWA Y\to Y}
(cf.\ \cite[\S2 A]{DroC} is an $A$-equivalence.

Such an $A$-equivalence induces a weak equivalence of \Ama[s] \w[,]{f\sb{\ast}:\fMA X\to\fMA Y} in the obvious sense.
\end{defn}

\begin{defn}\label{dfreema}
A \emph{free} \Fma is one of the form \w{\fMF B} for \w[.]{B\in\TsF}
\end{defn}

\begin{lemma}[cf.\ \protect{\cite[8.13]{BBlaC}}]\label{lfreema}
If $\fY$ is an \Fma and \w{\fMA B} is a free \Fma (for \w[),]{B\in\TsF}
there is a natural isomorphism
$$
\Phi:\map_{\MsF}(\fMF B,\fY)~\xra{\cong}~\fY\lin{B}~,
$$
\noindent with \w{\Phi(\fff)=\fff(\Id_{B})\in\fY\lin{B}_{0}} for any
\w[.]{\fff\in\Hom_{\MsF}(\fMF B,\fY)=\map_{\MsF}(\fMF B,\fY)_{0}}
\end{lemma}

\begin{proof}
This follows from the strong Yoneda Lemma for enriched categories (see
\cite[2.4]{GKellyEC}).
\end{proof}

Evidently, the values of any \Ama $\fX$ on the \emph{objects} of \w{\TsA} is completely determined by the single simplicial set \w[,]{Y:=\fX\lin{A}}
since all other values are obtained from $Y$ by \wref[.]{eqcoploop} This leads to the following

\begin{defn}\label{ddiscma}
A \emph{discrete} \Ama is a function \w[,]{\fX:\Obj(\TsA)\to\Sk}
written as before \w[,]{B\mapsto\fX\lin{B}} equipped with isomorphisms
\begin{myeq}[\label{eqdtheta}]
\theta_{A}:\fX\lin{\Sigma A}\to\Omega\fX\lin{A}
\end{myeq}
\noindent as in \wref[.]{eqcoploop}  The category of such is denoted
by \w[,]{\MsAd} and there is a forgetful functor \w[,]{U:\MsA\to\MsAd}
written \w[.]{\fX\mapsto\fX^{\delta}}
\end{defn}

\begin{mysubsection}{Extended \Ama[s]}\label{sema}
We can of course further extend the domain of definition of a discrete \Ama
to the collection \w{\C_{A}}of all objects of $\C$ obtained recursively from $A$ by pointed (homotopy) colimits. However, lifting such an ``extended discrete \Ama'' to
a simplicial functor on the full simplicial subcategory of $\C$ with objects \w{\C_{A}}  requires a great deal of additional structure, so we want to be able to restrict the type of colimits used to the minimum actually required for our purposes.  In particular, we need to restrict to a \emph{set} of objects in  \w[.]{\C_{A}}

In this connection we recall the following:
\end{mysubsection}

\begin{defn}\label{dsmall}
An object $A$ in a pointed simplicial model category $\C$ is called
$\lambda$-\emph{small}, for an ordinal $\lambda$, if the canonical map
\begin{myeq}[\label{eqsmall}]
\colim_{\alpha}\ \map_{\C}(B,X_{\alpha})~\to~
\map_{\C}(B,\colim_{\alpha}\,X_{\alpha})
\end{myeq}
\noindent is an isomorphism for any sequence
\w{(X_{\alpha})_{\alpha<\lambda}} in $\C$ and any object $B$ of the
form \w{B=\Sigma^{i}A} \wb[.]{i\geq 0}
\end{defn}

This always holds for any $A$ (for some limit ordinal $\lambda$) when
$\C$ is \w{\Sa} or \w{\Ta} (cf.\ \cite[\S 1.B.3]{DroC}); we shall
always assume the source \w{A\in\C} for any putative mapping space is
$\lambda$-small for some $\lambda$; the least such infinite ordinal
will be denoted by \w[.]{\lambda(A)} In particular, when \w{A\in\Ss} is a finite CW complex, we have \w[.]{\lambda(A)=\omega}

\begin{defn}\label{deama}
Given \w{A\in\C} and a set $\II$ of diagrams \w[,]{I\to\C} let \w{\TisA} denote the smallest sub-simplicial category of $\C$ containing $A$ and closed under pointed colimits over the diagrams in $\II$.

We shall always assume that all objects in \w{\TisA} are obtained from $A$ by iterated pointed colimits over diagrams from $\II$, and \w{\TisA} is closed under suspensions and coproducts of cardinality $<\kappa$, for some limit cardinal $\kappa$.

We define an \emph{extended \Ama} $\fX$ with respect to $\II$ to be a simplicial functor \w{\fX:(\TisA)\op\to\Sk} such that:
\begin{myeq}\label{eqcolim}
\fX\lin{B}~:=~\lim_{i\in I}\,\fX\lin{B_{i}}
\end{myeq}
\noindent whenever \w{B=\colim_{i\in I}B_{i}} for some diagram \w{I\to\C} in $\II$.
By assumption $\fX$ will in particular satisfy \wref[.]{eqcoploop}
The category of all such extended \Ama[s] will be denoted by \w[.]{\MsAI}
\end{defn}

%
%
\begin{lemma}\label{lsmallma}
If $\fX$ is an extended \Ama such that \w{\fX\lin{A}} is $\kappa$-small in \w[,]{\Sa} then $\fX$ is $\kappa$-small in \w[.]{\MsAI}
\end{lemma}

\begin{proof}
A map of extended \Ama[s] \w{\fff:\fX\to\fY} is determined by the map of pointed simplicial sets \w[,]{\fff\lin{A}:\fX\lin{A}\to\fY\lin{A}} because
by \wref{eqcoploop} we have \w[,]{\fff\lin{\coprod_{i}\,A_{i}}=\prod_{i}\,\fff\lin{A_{i}}}
and similarly for the other colimits in $\II$ by \wref[.]{eqcolim}
\end{proof}

\begin{defn}\label{ddema}
Given \w{A\in\C} and $\II$ as in \S \ref{deama}, a \emph{discrete extended \Ama}
(with respect to \w[)]{\TisA} is a function \w[,]{\fY:\TisA\to\Sk} written
\w[,]{B\mapsto\fY\lin{B}} which takes the colimits in \w{\TisA}
(of the diagrams in $\II$) to limits in \w{\Sa} (compare \S \ref{ddiscma}).

Note that such an $\fY$ is completely determined by the pointed simplicial
set \w{X:=\fY\lin{A}} (which we assume to be fibrant), with \w{\fY\lin{\colim_{I}B_{i}}:=\lim_{I}\fY\lin{B_{i}}} for every
diagram \w{I\to\C} in $\II$. We therefore denote this discrete extended \Ama $\fY$
by \w[.]{\fD{X}}

Letting $\OO$ denote the set of objects of \w[,]{\TisA} we define the \ww{\SaO}-category \w{\CX} to be the sub-simplicial category of \w{\Sa\op} whose objects are in the image of \w{\fD{X}} (cf.\ \S \ref{snot}).  Thus if
\w{B=\colim_{I}B_{i}} and \w{C=\colim_{J}C_{j}} are two objects in $\OO$, we have
\begin{myeq}\label{eqdema}
\begin{split}
\map_{\CX}(B,C)~:=&~\map_{\Sa}(\lim_{J}\fD{X}\lin{C_{j}},\,\lim_{I}\fD{X}\lin{B_{i}})\\
=&~\lim_{I}\map_{\Sa}(\lim_{J}\fD{X}\lin{C_{j}},\,\fD{X}\lin{B_{i}})~.
\end{split}
\end{myeq}
\end{defn}

By analogy with \cite[Lemma 1.4]{MayG} we have:

\begin{lemma}\label{ldema}
For any fibrant \w[,]{X\in\Sa} there is a natural one-to-one correspondence
between extended \Ama[s] \w{\fY\in\MsAI} with \w{\fY\lin{A}=X} and \ww{\SaO}-maps \w[.]{f:\TisA\to\CX}
\end{lemma}

\begin{proof}
Since all objects in $\OO$ are obtained from $A$ by iterated colimits, we see
from \wref{eqdema} that all the mapping spaces in \w{\CX} are obtained from
those of the form
\begin{myeq}\label{eqmaplimit}
\map_{\Sa}(\lim_{J}\fD{X}\lin{C_{j}},X)
\end{myeq}
\noindent by iterated limits. The action of
\w{\map_{\C}(A,\colim_{J}C_{j})} on \w{\fY\lin{\colim_{J}C_{j}}} is then given by
\begin{equation*}
\xymatrix@R=25pt{
\map_{\C}(A,\colim_{J}C_{j}) \ar[d]^{f} & \times &
\fY\lin{\colim_{J}C_{j}}~~~ \ar@{=}[d] \ar[rr] && ~~~\fY\lin{A} \ar@{=}[d]\\
\map_{\Sa}(\lim_{J}\fD{X}\lin{C_{j}},X) & \times &
\lim_{J}\fD{X}\lin{C_{j}}~~~ \ar[rr]^{\circ} && ~~~X
}
\end{equation*}
\noindent where the bottom map is composition in \w[.]{\CX}
\end{proof}

%
%
\sect{The Stover construction}
\label{csc}

In \cite[\S 2]{StoV}, Stover showed  how one could use a simple construction in order to encode the homotopy type of a (pointed connected) topological space \w{Y\in\Ta} by means of wedges of spheres: this was later used in \cite{DKStE} to define cofibrant replacements in the $E^{2}$-model category of simplicial sets.

Stover's basic idea was to start with a wedge of $n$-spheres \w{\bS{n}_{(f)}} of \wb[,]{n\geq 1} one for each map \w[.]{f:\bS{n}\to Y} If $f$ is null-homotopic, attach a cone \w{C\bS{n}} to \w{\bS{n}_{(f)}} for each null-homotopy \w{F:f\sim\ast} of $f$. The resulting space \w{LY} is homotopy equivalent to a wedge of spheres (of various dimensions), and is equipped with the obvious ``tautological'' map \w[.]{\var:LY\to Y} Iterating this yields a simplicial resolution of $Y$: that is, a simplicial space \w{\Vd} whose realization \w{\|\Vd\|} is weakly equivalent to the original $Y$, with each \w{V_{n}} homotopy equivalent to a wedge of spheres.  Moreover, the cardinality of this wedge (pointed coproduct) is bounded by the number of $0$- and $1$-simplices in the mapping spaces \w[.]{\mapa(\bS{n},Y)}

It turns out that the functor $L$ factors through the mapping algebra functor \w{\fMA} (for \w[),]{A=\bS{1}} and in fact \w[,]{\Vd} and thus \w[,]{\|\Vd\|} is determined up to homotopy by the mapping algebra \w{\fMA Y}
(cf.\ \cite[\S 9]{BBlaC}) \wh that is, by the loop space \w[,]{\Omega Y} with the action of mapping spaces between spheres on it (cf.\ \S \ref{egmapalg}). This is not very surprising, since this action includes inter alia the \ww{A_{\infty}} structure on \w[,]{\Omega Y} from which we can recover $Y$ up to weak equivalence by the methods of \cite{StaH,MayG}.

However, the Stover construction can be applied to spaces other than \w{\bS{1}} and its suspensions (although it needs to be modified when we start with a non-co-$H$-space $A$, since in this case it is not enough to encode the null-homotopies \wh we need all (pointed) homotopies). Thus we may hope that the \Ama structure on \w[,]{\mapa(A,Y)} when described properly in terms of the actual spaces appearing in the Stover constructions (rather than the strict pointed coproducts and suspensions of Definition \ref{dmapalg}), will allow us to recover \w{\CWA Y} from it.

This motivates the following series of definitions:

\begin{defn}\label{deso}
Let $A$ be a cofibrant object in $\C$, as in \S \ref{sass}, and $\kappa$ a cardinal. An \emph{elementary $A$-Stover object} (for $\kappa$) is any colimit \w[,]{B=\colim_{I} X} where $I$ is an indexing category consisting of a disjoint union of:
\begin{enumerate}\renewcommand{\labelenumi}{(\alph{enumi})~}
\item A set of objects $S$ of cardinality $<\kappa$;
\item For each pair \w[,]{(u,v)\in S\times S} a (possibly
  empty) set of objects \w{T\li{u,v}} of cardinality $<\kappa$;
\end{enumerate}

For each object \w[,]{t\in T\li{u,v}} $I$ has two morphisms \w{i_{u}:u\to t}
and \w[.]{i_{v}:v\to t}

The diagram $X$ has \w{X(u)\cong A} for each \w[,]{u\in S}
\w{X(t)\cong A\otimes\Delta[1]} for each \w[,]{t\in T\li{u,v}} where
\w{X(i_{u})} is identified with
\w{i_{0}:A\otimes\{0\}\hra A\otimes\Delta[1]} and \w{X(i_{v})}
identified with \w{i_{1}:A\otimes\{1\}\hra A\otimes\Delta[1]}
for any \w[.]{(u,v)\in S\times S}

An \emph{$A$-Stover object} (for $\kappa$) is a coproduct in $\C$ of a set of cardinality \www{<\kappa} of elementary \ww{\Sigma^{i}A}-Stover objects for $\kappa$, for various objects \w[.]{i\geq 0}
\end{defn}

\begin{remark}\label{reso}
Note that if \w{\C=\Sa} (or \w[),]{\Ta} the colimits described above are in fact (pointed) homotopy colimits. We shall assume this is true of the pointed model category $\C$ in which we are working; this will make the construction of Section \ref{chiv} homotopy invariant.

Moreover, since \w{\Sa} and \w{\Ta} are pointed, the simplicial operation \w{X\otimes K} is given by the half-smash of simplicial sets (or topological spaces):
\begin{myeq}[\label{eqhalfsmash}]
X\otimes K~=~X\rtimes K~:=~(X\times K)/(\ast\times K)
\end{myeq}
\noindent (which is itself pointed).
\end{remark}

\begin{mysubsection}{The generalized Stover comonad}\label{sesc}
Let \w{\F=\F\sp{\St}} be the set of elementary Stover objects in
$\C$ (for $\kappa$), and \w{\TssA} the full
sub-simplicial category of \w{\C\op} generated by \w{\F\sp{\St}} under
suspensions and coproducts of cardinality \www[.]{<\kappa}
An \Fma $\fX$ will be called a \emph{Stover \Ama}
(see \S \ref{dmapalg}), the category
Stover \Ama[s] will be denoted by \w[,]{\MsAs} and the discrete version
(\S \ref{ddiscma}) by \w[.]{\MsAsd} The corresponding mapping functors
are \w{\fMAs:\C\to\MsAs} and \w[,]{\fMAsd:\C\to\MsAsd} respectively.

This allows us to define the \emph{generalized Stover comonad}
\w{L_{A}:\C\to\C} by:
\begin{myeq}[\label{eqesc}]
L_{A}Y~:=~\coprod _{i\geq 0}\ \coprod_{\phi\in[\Sigma^{i}A,Y]}\ E_{\phi}~,
\end{myeq}
\noindent for any \w[,]{Y\in\C} where for each \w{B:=\Sigma^{i}A}
\wb[,]{i\geq 0} and homotopy class
\w{\phi\in[B,Y]=\pi_{0}\mapa(B,Y)} in \w[,]{\ho\C} the elementary $B$-Stover object \w{E_{\phi}=\colim_{I} X} is defined as in \S \ref{deso} with:

\begin{enumerate}
\renewcommand{\labelenumi}{(\alph{enumi})~}
\item The set $S$ consisting of all \w{f\in\map_{\C}(B,Y)_{0}} with
  \w[;]{[f]=\phi}
\item For each pair \w[,]{(f,g)\in S\times S} the set \w{T\li{f,g}}
  consisting of all homotopies
  \w{F\in\map_{\C}(B,Y)_{1}=\Hom(B\rtimes\Delta[1],Y)} with
  \w{d_{0}F=f} and \w[.]{d_{1}F=g}
\end{enumerate}
\end{mysubsection}

\begin{remark}\label{rescm}
When $A$ is a homotopy cogroup object (e.g., suspensions), we can make do with the original Stover construction, as defined in \cite[\S 2]{StoV}, where for any \Ama $\fY$:
\begin{myeq}[\label{eqstovc}]
L\fY~:=~\ \coprod _{i\geq 0} \
     \coprod _{\phi\in \fY\lin{\sigma^{i}A}_{0}} \ \colim \left(\sigma^{i}A\li{\phi}\xra{\inc}
                  (C\sigma^{i}A\li{\Phi})_{\Phi\in \fY\lin{\sigma^{i}A}_{1},
     d_{0}\Phi=\phi, d_{1}\phi=0}\right)~,
\end{myeq}
\noindent that is, we use only cones indexed by null-homotopies,
instead cylinders indexed by arbitrary homotopies (see also
\cite[\S 9]{BBlaC}). We identify the copy of \w{\sigma^{i}A} indexed by the $0$
map itself with $\ast$, so that the colimit of cones attached to it
becomes a coproduct of copies of \w[,]{\sigma^{i+1}A} which we identify
with the appropriate summands associated to \w[.]{\fY\lin{\sigma^{i}A}_{0}}

Note, however, that this will not work without the cogroup condition,
so that for arbitrary $A$ we need the generalized version defined above.
\end{remark}

\begin{defn}\label{dtdsalg}
Let $J$ be the category having two objects \w{j,j'} and two non-identity
maps \w[,]{\phi,\psi:j\to j'} and define a truncation functor
\w{\rho:\Ss\to\Set^{J}} by \w[.]{\rho K:=\{d_{0},d_{1}:K_{1}\to K_{0}\}}
A \emph{truncated discrete Stover \Ama} is a function \w[,]{\fX:\Obj(\TssA)\to\Set^{J}} equipped natural embeddings
\w{\fX\lin{\Sigma A}_{0}\to\fX\lin{A}_{1}} as in \wref[.]{eqdtheta}
\end{defn}

%
%
\begin{lemma}\label{lescm}
The functor \w{L_{A}:\C\to\C} is a comonad.
\end{lemma}

\begin{proof}
Note that \w{L_{A}} factors as the composite
\w[,]{\tLA\circ\rho\circ\fMAsd} since from \wwref{eqesc} \textit{ff.}
we see that \w{L_{A}Y} depends only on the $1$-skeleton of the
discrete Stover \Ama \w[.]{\fMAsd Y} It is readily verified that
\w{\tLA} is left adjoint to the truncated discrete Stover \Ama functor \w{\rho\fMAsd} (compare \cite[Proposition 9.12]{BBlaC}).
\end{proof}

\begin{remark}\label{rcounit}
The counit $\var$ for this comonad is the tautological map
\w[,]{\var_{Y}:L_{A}Y\to Y} which sends the copy of $B$ indexed by
$\phi$ to $Y$ via the map $\phi$, and the copy of \w{B\rtimes\Delta[1]}
indexed by $\Phi$ to $Y$ via $\Phi$ (in the notation of \S \ref{sesc}).
The comultiplication \w{\mu:L_{A}Y\to L_{A}L_{A}Y} may also be
described explicitly.
\end{remark}

\begin{mysubsection}{The hierarchy of \ma[s]}\label{shma}
So far we have described a hierarchy of notions of \ma[s,] each needed at a different stage in our reconstruction process:

\begin{enumerate}
\renewcommand{\labelenumi}{(\alph{enumi})~}
\item The truncated discrete Stover \Ama \w{\rho\fMAsd Y} contain the minimal amount of information needed to define the Stover construction \w[,]{L_{A}Y} and thus ``resolve'' $Y$ by wedges of half-suspensions of $A$.
\item The various discrete \Ama[s,] such as \w[,]{\fMAsd Y} can be recovered from \w{X:=\mapa(A,Y)} alone by means of appropriate (homotopy) limits.
\item Ordinary \Ama[s] such as \w{\fMA Y} (with variants such as the Stover \Ama \w[)]{\fMAs Y} contain a great deal of additional information about
\w[.]{X:=\mapa(A,Y)} However, these suffice to reconstruct $Y$ only for very special choices of $A$, such as wedges of spheres in \w{\Ta}  (see
\cite{BadzA,BDoraL,BBlaC} and \S \ref{slsr} below).
\item In general, in order to recover $Y$ (up to $A$-equivalence) we will have to allow additional (homotopy) colimits in define suitable \emph{extended} \Ama[s,] with respect to an extended theory \w[,]{\bT\sp{\II}\sb{A}} as in \S \ref{sema}. However, we have to be careful in making these choices, so as to guarantee that \w{\bT\sp{\II}\sb{A}} is small as an enriched category (as well as being as economical as possible in terms of the extra structure imposed on $X$), on the one hand \wh but at the same time sufficiently large to permit recovering $Y$. The necessary choices are described in \S \ref{sia} below.
\end{enumerate}

\end{mysubsection}

%
%
\sect{Monads and algebras}
\label{cmona}

Given a monad \w{\Theta:\C\to\C} over any category $\C$, we can define
algebras over $\Theta$ as in \cite[VI, \S 2]{MacLC}, as an alternative to ``universal algebras'', described in terms of operations and relations.
Unfortunately, the functor
\w{\fTAsd:=\fMAsd\circ\tLA\circ\rho:\MsAsd\to\MsAsd} is not
a monad on discrete Stover \Ama[s.] However, it still fits into the following framework:

\begin{defn}\label{dlalg}
Let \w{\TsF} be some sub-simplicial category of $\C$ as in \S \ref{dmapalg} (possibly discrete), such that \w[,]{\F\sp{\St}\subseteq\F} and let \w{\hr:\MsF\to\rho\MsAsd} be the induced forgetful functor from \Fma[s] to truncated discrete Stover \Ama[s] (for some fixed $\kappa$). Let \w{\fTF} denote the endofunctor \w[.]{\fMF\circ\tLA\circ\hr:\MsF\to\MsF}
A \ww{\fTF}-\emph{algebra} is an \Fma $\fX$ equipped with a map of \Fma[s] \w{h_{\fX}:\fTF\fX\to\fX} making
the diagram
\mydiagram[\label{eqalgebra}]{
\fTF\fTF\fX \ar[rr]^<<<<<<<<{\fMF\var_{\tLA\hr\fX}}
\ar[d]_{\fTF(h_{\fX})} && \fTF\fX \ar[d]_{h_{\fX}} \\
\fTF\fX \ar[rr]_{h_{\fX}} && \fX
}
\noindent commute, with:
\begin{myeq}[\label{eqsplit}]
\hr h_{\fX}\circ\eta_{\fX}~=~\Id_{\hr\fX}~,
\end{myeq}
\noindent where

\begin{myeq}[\label{equnit}]
\eta_{\fX}:\hr\fX~\llra~\hr\fTF\fX~=~\hr\fMF\tLA\hr\fX~=~\rho\fMAsd\tLA\hr\fX
\end{myeq}
\noindent is the unit of the monad \w{\rho\circ\fMAsd\circ\tLA:\rho\MsAsd\to\rho\MsAsd} on truncated discrete Stover \Ama[s.]
\end{defn}

\begin{remark}\label{rlalg}
Note that if \w{\fX=\fMF Y} for some \w[,]{Y\in\C} then:
\mydiagram[\label{eqcoalgebra}]{
L_{A}L_{A}Y \ar[rr]^<<<<<<<<{\var_{\tLA\hr\fX}}
\ar[d]_{L_{A}(\var_{\fX})} &&  L_{A}Y \ar[d]_{\var_{Y}} \\
L_{A}Y \ar[rr]_{\var_{Y}} && Y
}
\noindent commutes by Lemma \ref{lescm} and \cite[VI, \S 1]{MacLC}. Thus
$$
h_{\fX}:=\fMF(\var_{Y}):\fMF L_{A}Y\to\fMF Y
$$
\noindent indeed makes \wref{eqalgebra} commute, by applying \w{\fMF} to
\wref[;]{eqcoalgebra} and \wref{eqsplit} also holds (cf.\ \cite[VI, \S 2]{MacLC}).
\end{remark}

%
%
\begin{prop}\label{ptalg}
For \w{\F\sp{\St}\subseteq\F} as above, every \Fma $\fX$ has a natural
\ww{\fTF}-algebra structure map of \Fma[s] \w[.]{h_{\fX}:\fTF\fX\to\fX}
\end{prop}

\begin{proof}
First, we show that there is a natural map
\w[,]{h=h_{\fX}:\fTF\fX\to\fX} determined by the
pre-compositions with maps in \w[:]{\TssA} since \w{B:=\tLA\rho\fX} is
in \w[,]{\TssA\subseteq\TsF} the \Fma \w{\fMF\tLA\hr\fX} overlying \w{\fTF\fX} is free, so by Lemma \ref{lfreema}, in order to  define $h$ as a
map of \Fma[s] we just need to choose an element in \w[.]{\fX\lin{B}_{0}}
Since $B$ is a coproduct, by \wref[,]{eqesc} we need elements \w{h_{\phi}}
in \w{\hr\fX\lin{E_{\phi}}_{0}} for each \w[,]{\phi\in\pi_{0}\hr\fX\lin{A}}
and since \w{E_{\phi}} is in turn a colimit, \w{h_{\phi}} is
determined by the tautological elements \w{f\in\hr\fX\lin{A'}_{0}=\fX\lin{A}_{0}} for each $f$
representing $\phi$.

To check that \wref{eqalgebra} commutes, note that
\w{C:=\tLA\hr\fTF\fX} is again in \w[,]{\TssA\subseteq\TsF} so the \Fma \w{\fTF\fTF\fX} is again free, and we need only check that
\w{\fTF\fTF\fX\lin{C}} is sent to the same place in \w[.]{\fX\lin{C}} Again,
$C$ is a coproduct of objects of the form \w[,]{E_{\phi}} which are in turn
colimits of objects $B$ or \w{B\rtimes\Delta[1]} with \w[.]{B=\Sigma^{k}A}
Each such $B$ is indexed by a map \w{f:B\to \bigvee_{i}~E_{\phi_{i}}\subseteq\tLA\hr\fX}
(or \w[),]{F:B\rtimes\Delta[1]\to \bigvee_{i} E_{\phi_{i}}} with \w{E_{\phi_{i}}}
itself a colimit of objects \w{B_{ik}=\Sigma^{n_{ik}}A} (or \w[)]{B_{ik}\rtimes\Delta[1]}
indexed by ``formal maps'' \w{f_{ik}\in\fX\lin{B_{ji}}} (or
\w[).]{F_{ik}\in\fX\lin{B_{ik}\rtimes\Delta[1]}} Together the maps \w{f_{ik}}
and \w{F_{ik}} (for all $j$) induce a ``formal map''
\w{g_{i}\in\fX\lin{E_{\phi_{i}}}} (for each $i$), and thus a single element
\w[.]{g\in\fX\lin{\bigvee_{i} E_{\phi_{i}}}}

Now if \w{j:B\to E_{\phi}\hra C} is the structure map for the colimit, followed
by the inclusion, thought of as a $0$-simplex in \w[,]{(\fTF\fTF\fX)\lin{B}}
then \w{\fMF\var_{\tLA\hr\fX}} (the top horizontal map in
\wref[)]{eqalgebra} takes $j$ to the map
\w[,]{f:B\to \bigvee_{i} E_{\phi_{i}}} thought of as a $0$-simplex in
\w[.]{(\fMF\tLA\hr \fX)\in{B}} Then \w{h_{\fX}} (the right vertical
map in \wref[)]{eqalgebra} takes $f$ to \w[,]{g\circ f\in\fX\lin{B}}
since $h$ is a map of Stover \Ama[s] by construction.

On the other hand, \w{\fTF(h_{\fX})} (the left vertical map in
\wref[)]{eqalgebra} takes $j$ to the map \w{\iota_{B}} identifying $B$ with
the copy of $B$ in the colimit \w{\tLA\hr\fX} indexed by the map
\w{g\circ f\in\fX\lin{B}} (by definition of the functor \w[,]{\fMF} and
thus of \w[),]{\fTF} and then \w{h_{\fX}} (the bottom horizontal map in
\wref[)]{eqalgebra} takes \w{\iota_{B}} to \w{g\circ f\in\fX\lin{B}}
by the definition above. Thus we see that \wref{eqalgebra} commutes on generators, and since all maps are maps of \Fma[s,] the diagram commutes.

To verify \wref[,]{eqsplit}] by adjunction it suffices to check that
$$
\tLA\hr h_{\fX}\circ\tLA\eta_{\fX}~=~\Id_{\tLA\hr\fX}~;
$$
\noindent indeed,  given \w[,]{f\in\hr\fX\lin{B}_{0}}
\w{\tLA\eta_{\fX}} takes the copy \w{B_{f}} of $B$ in \w{\tLA\hr\fX} isomorphically to the copy of \w{B_{\iota}} in \w{\tLA\rho\fMAsd\tLA\hr\fX} indexed by the inclusion \w[,]{\iota:B_{f}\hra \tLA\hr\fX} which
\w{\tLA\hr h_{\fX}} maps back to \w{B_{f}} by construction. Similarly for
\w[.]{F\in\hr\fX\lin{B}_{1}}
\end{proof}

\begin{remark}\label{rptalg}
The converse of Proposition \ref{ptalg} is also true, but as we shall not need this fact, we omit the proof.
\end{remark}

%
%
\sect{The inductive colimit construction}
\label{cicc}

We now describe a transfinite inductive process for recovering $Y$ (up to $A$-equivalence) from the mapping space \w[,]{X=\map(A,Y)} equipped with the extra structure defined in the previous sections.

\begin{mysubsection}{Induction assumptions}
\label{sia}
For any ordinal \w[,]{\alpha<\lambda} at the $\alpha$-th stage of the
induction we assume that we have constructed the previous
approximations \w{\qY{\beta}\in\C} for \w{\beta<\alpha} and chosen a cardinal \w[,]{\kappa_{\alpha}} as well as a set \w{\qF{\alpha}} of objects in $\C$, containing \w[,]{\bigcup_{\beta<\alpha}\,\qF{\beta}} all $A$-Stover objects for \w{\kappa_{\alpha}} (\S \ref{deso}), and closed under:

\begin{enumerate}
\renewcommand{\labelenumi}{(\alph{enumi})~}
\item Suspensions and half-suspensions (cf.\ \wref[):]{eqhalfsmash}
\item Coproducts of cardinality $\leq\kappa_{\alpha}$;
\item Taking \w{D'} in the functorial factorization \w{C\to D'\to D} of any map \w{C\to D} in \w{\qf{\alpha}} as a cofibration followed by an acyclic fibration.
\item Pushouts of diagrams \w{Z\leftarrow X\xra{i}Y}
when $i$ is a cofibration between $A$-Stover objects for
\w{\kappa_{\beta}}  and $Z$ is in \w[,]{\qF{\beta}} for \w[.]{\beta<\alpha}
\item Colimits of sequences of cofibrations in \w{\bigcup_{\beta<\alpha}\,\qF{\beta}}
of length $<\alpha$.
\end{enumerate}

Note that all the limits in question are homotopy colimits. and all objects are cofibrant (see Remark \ref{reso}).
\end{mysubsection}

\begin{defn}\label{dalphaext}
For any ordinal \w[,]{\alpha\leq\lambda} we let \w{\CasA{\alpha}} denote the sub-simplicial category of $\C$ with object set \w[,]{\qF{\alpha}} \w{\MsqF{\alpha}} the corresponding category of extended \Ama[s] 
(cf.\ \S \ref{deama}), and \w{\fMqF{\alpha}:\C\to\MsqF{\alpha}} the mapping algebra functor.
\end{defn}

\begin{remark}\label{rcardinal}
The cardinal \w{\kappa_{\alpha}} is determined inductively by letting
\w{\kappa_{0}:=|X|^{+}} be the least cardinal greater than the cardinality of the
truncated simplicial set \w[,]{\rho X} and thus a bound on the coproducts appearing in the $A$-Stover construction \w{\tLA\rho\fX} (in the notation of the proof of Lemma \ref{lescm}). In general, \w{\kappa_{\alpha+1}} is the least upper bound on the cardinality of \w{\rho\fMA Y} for all \w[.]{Y\in\qF{\alpha}} For a limit ordinal $\alpha$ we set \w[.]{\kappa_{\alpha}:=\sup_{\beta<\alpha}\,\kappa_{\beta}}
We assume for simplicity that \w[.]{\lambda\leq \kappa_{0}}
\end{remark}

\begin{remark}\label{rrealizable}
Of course, any \emph{realizable} \Ama \w{\fMA Z} extends uniquely
to an \qFma{\beta}\ \w{\fMqF{\beta} Z} for any
\w[;]{\beta\leq\lambda} in particular, we write \w{\fTqF{\beta}} for
\w{\fMqF{\beta}\circ\tLA:\rho\MsA\to\MsqF{\beta}} (no longer a monad).

The main additional hypothesis needed to realize the original $\fX$ is
that it can be extended to an \qFma{\beta}\ $\tfX$ for each
\w[.]{\beta\leq\alpha} Of course, this is trivially true for the
\emph{discrete} \Ama $\fX$, and thus the \ww{\fTA}-algebra structure
map \w{h_{\fX}:\fTA\fX\to\fX} extends canonically to
\w[.]{\tet{\fX}:\fTqF{\beta}\fX\to\tfX}

We can then make sense of the requirement that we have a map of
\qFma{\beta}s \w{\qf{\beta}:\fTqF{\beta}\qY{\beta}\to\tfX} and a map
\w{\qp{\beta}:\tLA\rho\fX\to\qY{\beta}} in $\C$ for each
\w[,]{\beta\leq\alpha} as well as a cofibration
\w{\qi{\beta}:\qY{\beta}\to\qY{\beta+1}} in $\C$ such that
\begin{myeq}\label{eqexpo}
\qf{\beta+1}\circ\fMqF{\beta+1}\qi{\beta}~=~\qf{\beta}\hs\text{and}\hs
\qf{\beta+1}\circ\fMqF{\beta+1}\qp{\beta+1}~=~\tet{\fX}
\end{myeq}
\noindent for each \w[,]{\beta<\alpha} and
\begin{myeq}\label{eqexp}
\qf{\beta}\circ\fMqF{\beta}\var_{\qY{\beta}}~=~
\tet{\fX}\circ\fTqF{\beta}\qf{\beta}
\end{myeq}
\noindent for each \w[.]{\beta\leq\alpha}
\end{remark}

\begin{mysubsection}{The inductive construction}
\label{sic}
We start with \w{\qY{0}:=\ast\in\C} as the $0$-th approximation to
$Y$, with \w[.]{\qF{0}=\{A\}} Assume by induction that we have obtained the above structure for some \w[.]{\alpha<\lambda}  Now factor
\w{\tLA\rho\qf{\alpha}:\tLA\rho\fMA\qY{\alpha}\to\tLA\rho\fX}
(functorially) as
\mydiagram[\label{eqfactorize}]{
\tLA\rho\fMA\qY{\alpha}~~\ar@{^{(}->}[rr]^-{\wqf{\alpha}} &&
~~\wLA{\fX}~~ \ar@{->>}[rr]_{\simeq}^{\ell\q{\alpha}} &&
~~\tLA\rho\fX
}
\noindent with \w{\wqf{\alpha}} a cofibration and \w{\ell\q{\alpha}} a trivial fibration. Because \w{\tLA\rho\fX} is cofibrant, \w{\ell\q{\alpha}} has a section \w{j\q{\alpha}:\tLA\rho\fX\to\wLA{\fX}} (which is also a weak equivalence), with
\begin{myeq}\label{eqretract}
\ell\q{\alpha}\circ j\q{\alpha}~=~\Id_{\tLA\rho\fX}~.
\end{myeq}

We now define \w[,]{\qY{\alpha+1}} \w[,]{\qp{\alpha+1}} and \w{\qi{\alpha}} by the pushout diagram:
\mydiagram[\label{eqmapo}]{
\ar @{} [drrr] |>>>>>>>{\framebox{\scriptsize{PO}}}
L_{A}\qY{\alpha}~\ar[d]^{\var_{\qY{\alpha}}}~~ \ar@{=}[r] &
~\tLA\rho\fMA\qY{\alpha} \ar@{^{(}->}[rr]^-{\wqf{\alpha}} &&
\wLA{\fX} \ar[d]^{\qp{\alpha+1}}\\
\qY{\alpha}~~  \ar@{^{(}->}[rrr]_{\qi{\alpha}} &&& ~~\tqY{\alpha+1}
}
\noindent in $\C$ (which is thus a homotopy push out).

Since all but the lower left corner of \wref{eqmapo} are in
\w[,]{\qF{\alpha}\subseteq\qF{\alpha+1}} the pushout \w{\qY{\alpha+1}}
is in \w[,]{\qF{\alpha+1}} by definition.
Therefore, if we apply the \qFma{\alpha+1}-functor \w{\fMqF{\alpha+1}}
to \wref[,]{eqmapo} we have a pushout diagram in
\w[,]{\M\sb{\qF{\alpha+1}}} by Lemma \ref{lfreema}.

Moreover, by \wref{eqexpo} the following square commutes in \w[:]{\MsqF{\alpha}}
\mydiagram[\label{eqmasq}]{
\fTqF{\alpha}\fMqF{\alpha}\qY{\alpha} \ar[d]^{\fMqF{\alpha}\var_{\qY{\alpha}}}
\ar[rr]^{\fMqF{\alpha}\wqf{\alpha}} \ar@/^{2.5pc}/[rrrrr]^{\fTqF{\alpha}\qf{\alpha}} &&
~\fMqF{\alpha}\wLA{\fX}~~ \ar[rr]^{\fMqF{\alpha}\ell\q{\alpha}} &&
\fTqF{\alpha}\fX \ar@{=}[r] &
\fL_{A}\fX \ar[d]^{\tet{\fX}}\\
\fMqF{\alpha}\qY{\alpha}  \ar[rrrrr]_{\qf{\alpha}} &&&&& \tfX,
}
\noindent and since all but the lower right corner are free
\qFma{\alpha}s, \wref{eqmasq} extends automatically to a commuting
square in \w{\MsqF{\alpha+1}} (see \S \ref{sema}).

Therefore, there is a dotted map of \qFma{\alpha+1}s making
the following diagram commute:
\mydiagram[\label{eqinmapo}]{
\fTqF{\alpha+1}\fMqF{\alpha+1}\qY{\alpha}
\ar[d]^{\fMqF{\alpha+1}\var_{\qY{\alpha}}}
\ar[rrr]^-{\fMqF{\alpha+1}\wqf{\alpha}} &&&
\fMqF{\alpha+1}\wLA{\fX} \ar@/^/[ddrr]^{\tet{\fX}\circ\fMqF{\alpha+1}\ell\q{\alpha}}
\ar[d]_{\fMqF{\alpha+1}\qp{\alpha+1}}\\
\fMqF{\alpha+1}\qY{\alpha}  \ar[rrr]^{\fMqF{\alpha+1}\qi{\alpha}}
\ar@/_/[drrrrr]_{\qf{\alpha}} &&&
\fMqF{\alpha+1}\qY{\alpha+1} \ar@{.>}[drr]_<<<<<{\qf{\alpha+1}} & \\
&&&&& \hso\tfX~.
}
\noindent This completes the induction step, by defining \w[.]{\qf{\alpha+1}}
\end{mysubsection}

\begin{mysubsection}{Verifying the induction hypothesis}
\label{svih}
By construction \wref{eqexpo} holds for \w[;]{\alpha+1} to see that \wref{eqexp} holds, too, consider the following diagram of \qFma{\alpha+1}s:
$$
\xymatrix@R=25pt{
%
\fTqF{\alpha+1}\fMqF{\alpha+1}\qY{\alpha+1}
\ar[ddd]^{\fMqF{\alpha+1}\var_{\qY{\alpha+1}}}
\ar[rrrrrr]^{\fTqF{\alpha+1}\qf{\alpha+1}}
&&&&&& \fTqF{\alpha+1}\fX \ar[ddd]^{\tet{\fX}} \\
&& \fTqF{\alpha+1}\fMqF{\alpha+1}\wLA{\fX}
\ar[d]^{\fMqF{\alpha+1}\var_{\wLA{\fX}}}
\ar[rrr]^{\fTqF{\alpha+1}\fMqF{\alpha+1}\ell\q{\alpha}}
\ar[llu]^{\fTqF{\alpha+1}\fMqF{\alpha+1}\qp{\alpha+1}} &&&
\fTqF{\alpha+1}\fTqF{\alpha+1}\fX \ar[ru]_{\fTqF{\alpha+1}\tet{\fX}}
\ar[d]^{\fMqF{\alpha+1}\var_{\tLA\rho\fX}} & \\
&& \fMqF{\alpha+1}\wLA{\fX} \ar[lld]_{\fMqF{\alpha+1}\qp{\alpha+1}}
\ar[rrr]^{\fMqF{\alpha+1}\ell\q{\alpha}} &&&
\fTqF{\alpha+1}\fX \ar[rd]^{\tet{\fX}}  & \\
\fMqF{\alpha+1}\qY{\alpha+1} \ar[rrrrrr]^{\qf{\alpha+1}} &&&&&& \hso\tfX~.
}
$$
\noindent The top square commutes by applying \w{\fTqF{\alpha+1}} to
the right triangle of \wref[;]{eqinmapo} the right square is
\wref{eqalgebra} for \w[,]{\fTqF{\alpha+1}} which commutes by
Proposition \ref{ptalg}, since \w[.]{\F\sp{\St}\subseteq\qF{\alpha+1}}  The left and middle squares commute by naturality of $\var$; and the bottom square is the right triangle of \wref[.]{eqinmapo} Therefore, the outer square commutes.

If $\alpha$ is a limit ordinal, we let
\w[,]{\qY{\alpha}:=\colim_{\beta<\alpha}\qY{\beta}} with the structure
maps \w{\qi{\beta}:\qY{\beta}\to\qY{\alpha}} for each
\w[.]{\beta<\alpha} Note that by construction
\w[,]{\tLA\rho\fX=\colim_{\alpha<\beta}\tLA\rho\fX}
since \w[,]{\qF{\alpha}=\bigcup_{\alpha<\beta}\,\qF{\beta}}
so we may set \w{\qp{\alpha}:=\colim_{\alpha<\beta}\,\qp{\beta}}
and \w[.]{\qf{\alpha}:=\colim_{\alpha<\beta}\,\qf{\beta}}
\end{mysubsection}

%
%
\sect{Realizing mapping algebras}
\label{crma}

Evidently, if \w{X=\mapa(A,Y)\in\Sa} is a pointed mapping space
out of $A$, we can use \wref{eqcoploop} to define the \Ama
\w{\fX=\fMA Y} with \w[.]{\fX\lin{A}=X} More generally, $\fX$
can be made into an extended \Fma[,] where $\F$ is the closure of \w{\{A\}}
in $\C$ with respect to any set of homotopy colimits constructed from $A$.

We now show that the \Ama[-structure] on $X$, combined with the constructions of Section \ref{cicc} indeed allows us to recover $Y$ (up to $A$-equivalence):

%
%
\begin{thm}\label{treal}
Let $A$ be a cofibrant $\lambda$-small object in a model
category $\C$ as in \S \ref{sass}, and $X$ a pointed simplicial
set. If $X$ can be equipped with an extended \qFma{\lambda} structure
$\fX$ for the family \w{\F=\qF{\lambda}} defined in Section \ref{cicc}, then the map of \Fma[s] \w{\fff:=\qf{\lambda}:\fMqF{\lambda}\qY{\lambda}\to\hfX}
constructed there induces a weak equivalence \w[.]{X\simeq\map_{\C}(A,Y)}
\end{thm}

\begin{proof}
It suffices to show that
\w{\fff_{\#}:\pi_{0}((\fMA\qY{\lambda})\lin{\Sigma^{i}A)}\to
\pi_{0}(\hfX\lin{\Sigma^{i}A})} is bijective for any \w[\vsm.]{i\geq 0}

\noindent \textbf{(a)}~~ We can use the unit \w{\eta_{\fX}:\rho\fX\to\rho\fMA\tLA\rho\fX} of
\wref{equnit} to define:
$$
\qs{\alpha+1}~:=~\rho\fMA\qp{\alpha+1}\circ\rho\fMA\qp{\alpha+1} j\q{\alpha+1}\circ\eta_{\fX}:~\rho\fX~\to~\rho\fMA\qY{\alpha+1}
$$
\noindent for any \w[.]{\alpha<\lambda}

Since \w[,]{\F\sp{\St}\subseteq\qF{\lambda}} $\hfX$ is an \ww{\fTF}-algebra by Proposition \ref{ptalg}. We therefore have a commuting diagram:
\mydiagram[\label{eqsplitoff}]{
\rho\fX=\hr\hfX \ar[r]^{\eta_{\fX}} \ar[rrdd]_{\Id} & \rho\fMAsd\tLA\rho\fX \ar@{=}[r] &
\hr\fMqF{\alpha+1}\tLA\rho\fX \ar[rr]^{\hr\fMqF{\alpha+1}j\q{\alpha+1}} \ar[d]_{\Id} &&
\hr\fMqF{\alpha+1}\wLA{\fX} \ar[lld]_{\hr\fMqF{\alpha+1}\ell\q{\alpha+1}}
\ar[d]^{\hr\fMqF{\alpha+1}p\q{\alpha+1}} \\
&& \hr\fMqF{\alpha+1}\tLA\rho\fX \ar[d]^{\hr h_{\fX}} && \hr\fMqF{\alpha+1}Y\q{\alpha+1} \ar[lld]^{\hr\qf{\alpha+1}} \\
&& \rho\fX=\hr\hfX
}
\noindent in which the left triangle commutes by \wref[,]{eqsplit}
the upper right triangle commutes by \wref[,]{eqretract} and the lower right square commutates by \wref[.]{eqinmapo} This shows that
\begin{myeq}[\label{eqsplitter}]
\hr\qf{\alpha+1}\circ\qs{\alpha+1}=\Id_{\rho\fX}~\vsm.
\end{myeq}

\noindent \textbf{(b)}~~ Thus we have a diagram
\mydiagram[\label{eqladder}]{
\hr\fMA\qY{0}~\ar[r]^{\hr\fMA\qi{0}} \ar[d]_{\rho\qf{0}} &
\hr\fMA\qY{1}~\ar[r]^<<<<<<<{\hr\fMA\qi{1}} \ar[d]_{\hr\qf{1}} &~~\dotsc\hspace*{5mm}
\hr\fMA\qY{\alpha}~\ar[r]^{\hr\fMA\qi{\alpha}} \ar[d]_{\hr\qf{\alpha}} &
\hr\fMA\qY{\alpha+1}~\ar[r]^<<<<<<<{\hr\fMA\qi{\alpha+1}}
\ar[d]_{\hr\qf{\alpha+1}} & \dotsc \\
\hr\hfX\ar@{=}[r] \ar[ru]^{\qs{1}} &
\hr\hfX\ar@{=}[r] \ar[ru]^{\qs{2}} & ~~\dotsc~~
\hr\hfX\ar@{=}[r] \ar[ru]^{\qs{\alpha+1}} &
\hr\hfX\ar@{=}[r] \ar[ru]^{\qs{\alpha+2}} & \dotsc
}
\noindent of truncated extended \Ama[s] in which the squares and lower triangles commute. By Lemma \ref{lsmallma} we then obtain a map
$$
s^{B}~=~\qs{\lambda}\lin{B}:\rho\fX\lin{B}\to\rho\fMA\qY{\lambda}\lin{B}
$$
for each \w{B=\Sigma^{i}A} \wb{0\leq i<\infty} with \w[.]{\rho\fff\lin{B}\circ s^{B}=\Id_{\rho\fX\lin{B}}}
This shows that \w{\fff_{\#}} is surjective\vsm.

\noindent \textbf{(c)}~~Since \w{\rho\fMA\qY{\alpha}\lin{B}} is just the
$1$-truncation of \w[,]{\map(B,\qY{\alpha})} by \wref{eqsmall} we know
$$
\rho\fMA\qY{\lambda}\lin{B}~=~
\colim_{\alpha<\lambda}\, \rho\fMA\qY{\alpha}\lin{B}~,
$$
\noindent and thus
\w[.]{\fff_{\#}=\colim_{\alpha<\lambda}\,\qf{\alpha}_{\#}:=
\colim_{\alpha<\lambda}\,\pi_{0}\rho\qf{\alpha}\lin{B}}
Therefore, if \w{\fff_{\#}(\gamma)=\fff_{\#}(\gamma')} for some
\w[,]{\gamma,\gamma'\in[B,\qY{\lambda}]} then there is an \w{\alpha<\lambda}
and \w[,]{\gamma_{\alpha},\gamma'_{\alpha}\in[B,\qY{\alpha}]} represented by
\w{g,g':B\to\qY{\alpha}} in \w[,]{\rho\fMA\qY{\alpha}\lin{B}_{0}} with
\w[.]{\qf{\alpha}_{\#}(\gamma_{\alpha})=\qf{\alpha}_{\#}(\gamma'_{\alpha})}
In other words, there is a $1$-simplex \w{\sigma\in\fX\lin{B}_{1}}
with \w{d_{0}\sigma=\qf{\alpha}(g)} and \w[.]{d_{1}\sigma=\qf{\alpha}(g')}

From the description in \S \ref{rlalg} we see that
\w[,]{\eta_{\rho\fX}(\sigma)=i_{B_{\sigma}}\in(\fL_{A}\fX)\lin{B}_{1}}
corresponding to the inclusion
\w[,]{B_{\sigma}\rtimes\Delta[1]\hra\tLA\fX} with
\w{d_{0}(i_{B_{\sigma}})=i_{B_{\qf{\alpha}(g)}}=\eta_{\rho\fX}(\qf{\alpha}(g))}
and \w[.]{d_{1}(i_{B_{\sigma}})=
i_{B_{\qf{\alpha}(g')}}=\eta_{\rho\fX}(\qf{\alpha}(g'))}
Similarly \w{\eta_{\rho\fMA\qY{\alpha}}(g)=
i_{B_{g}}\in(\fL_{A}\fMA\qY{\alpha})\lin{B}_{0}}
and
\w[,]{\eta_{\rho\fMA\qY{\alpha}}(g')=i_{B_{g'}}\in
(\fL_{A}\fMA\qY{\alpha})\lin{B}_{0}}
and all these $0$-simplices match up in the following diagram (compare
\wref[):]{eqinmapo}
\mydiagram[\label{eqinmp}]{
i_{B_{g}}\in\fMA L_{A}\qY{\alpha}
\ar[d]^{\fMA\var_{\qY{\alpha}}}
\ar[rrr]^-{\fMA\tLA\rho\qf{\alpha}} &&&
i_{B_{\qf{\alpha}(g)}}\in
\fMA\tLA\rho\fX \ar@/^/[ddr]^{h_{\fX}} \ar[d]_{\fMA\qp{\alpha+1}}\\
g\in\fMA\qY{\alpha}  \ar[rrr]^{\fMA\qi{\alpha}}
\ar@/_/[drrrr]_{\qf{\alpha}} &&&
\qi{\alpha}_{\ast}(g)\in\fMA\qY{\alpha+1}
\ar@{.>}[dr]_<<<<<{\qf{\alpha+1}} & \\
& & & & \hso \qf{\alpha}(g)\in\fX~,
}
\noindent and similarly for \w[.]{g'} If we now set
\w[,]{\tau:=\fMA\qp{\alpha+1}(i_{B_{\sigma}})
\in\fMA\qY{\alpha+1}\lin{B}_{1}}
then evidently \w{d_{0}\tau=\qi{\alpha}_{\ast}(g)} and
\w[,]{d_{1}\tau=\qi{\alpha}_{\ast}(g')} showing that
\w[,]{\qi{\alpha}_{\#}(\gamma_{\alpha})=
\qi{\alpha}_{\#}(\gamma'_{\alpha})}
so \emph{a fortiori} \w[\vsm .]{\gamma=\gamma'\in[B,\qY{\lambda}]}

Thus \w{\fff_{\#}} is bijective, so \w{\fff:\fMA\qY{\lambda}\to\fX} is
a weak equivalence of \Ama[s.] We have thus shown in particular that
the original space $X$ is weakly equivalent to
\w[.]{\map_{\C}(A,\qY{\lambda})}
\end{proof}

Compare \cite[\S 2, Proposition B.1]{DroC}.

%
%
\sect{A homotopy invariant recognition principle}
\label{chiv}

The recognition and recovery procedures described in Section \ref{crma} are 
rigid, inasmuch as they require that the structure on a given space $X$ 
hold on the nose in order to deduce that \w{X\simeq\map_{\C}(A,Y)} for some \w[.]{Y\in\C} We now describe a lax version of the latter, which is homotopy invariant in the sense that if the recognition principle holds for $X$, 
it also holds for any \w[.]{X'\simeq X}

\begin{defn}\label{dleat}
For any ordinal \w[,]{\alpha\leq\lambda} let \w{\CasA{\alpha}} denote the simplicial subcategory of $\C$ with object
set \w{\OO:=\qF{\alpha}} defined in \S \ref{sia}.
A \emph{lax ($\alpha$-extended) theory} (for $A$) is any cofibrant \ww{\SO}-category \w{\hCasA{\alpha}} weakly equivalent to \w{\CasA{\alpha}}
(in the model category on \w{\SOC} defined in \cite[\S 1]{DKanSR}).

A \emph{lax ($\alpha$-extended) \Ama} is a simplicial functor
\w{\hfX:(\hCasA{\alpha})\op\to\Sk} such that on objects \w{\OO=\qF{\alpha}}
$\hfX$ is isomorphic to the discrete extended \Ama \w{\fD{X}} of \S \ref{ddema}, for \w[.]{X:=\hfX\lin{A}} In this case we also say that $\hfX$ constitutes a
\emph{lax ($\alpha$-extended) \Ama structure} on $X$.

The category of lax $\alpha$-extended \Ama[s] is denoted by \w[.]{\MsAa{\alpha}} 
\end{defn}

\begin{mysubsection}{Constructing lax extended $A$-theories}
\label{slat}
Lax $\alpha$-extended theories for $A$ can be constructed explicitly in several ways:

First, consider the category \w[,]{\Fd\CasA{\alpha}} enriched in
bisimplicial sets, obtained by iterating the free category comonad
\w{\cF:=F\circ U} on each simplicial dimension of \w[,]{\CasA{\alpha}}
and let \w{\hCasA{\alpha}} be the cofibrant \ww{\SO}-category obtained from
\w{\Fd\CasA{\alpha}} by taking the diagonal of each bisimplicial set. Note
that \w{\hCasA{\alpha}} is free in each simplicial dimension.

A minimal version \w{\tCasA{\alpha}} of this construction may be obtained as
follows: start with \w{(\tCasA{\alpha})_{0}:=\F\pi_{0}\CasA{\alpha}}
which we can think of as a subcategory of \w{(\hCasA{\alpha})_{0}}
by choosing a section \w{\sigma:\pi_{0}\CasA{\alpha}\to(\CasA{\alpha})_{0}}
(that is, a representative \w{\sigma[\phi]:B\to B'} for each homotopy class \w{\phi\in[B,B']} and \w[).]{B,B'\in\qF{\alpha}}

Next, choose a representative \w{\sigma[\Phi]}
for each homotopy class of homotopies between the chosen maps in
\w{(\tCasA{\alpha})_{0}\subseteq\hCasA{\alpha})_{0}} and composites thereof,
and let \w{(\tCasA{\alpha})_{1}} be the free category generated by these
representatives \w{\sigma[\Phi]} and \w[.]{s_{0}(\tCasA{\alpha})_{0}}
Proceeding similarly for higher homotopies, we obtain a minimal model
\w{\tCasA{\alpha}} for \w[.]{\hCasA{\alpha}}

We can describe \w{\tCasA{\alpha}} more formally by using a section
\w{\sigma:\pi_{0}\CasA{\alpha}\to(\CasA{\alpha})_{0}} to define
\w{(\tCasA{\alpha})_{0}} as above, but then noting that this determines
$$
\map_{\tCasA{\alpha}}(\Sigma B, B')_{0}~=~\Omega\map_{\tCasA{\alpha}}(B, B')_{0}~
\subseteq~\map_{\tCasA{\alpha}}(B, B')_{1}
$$
\noindent for any \w[,]{B,B'\in\qF{\alpha}} which can then be used to
determine \w[,]{\map_{\tCasA{\alpha}}(B, B')_{1}} and similarly for the higher
simplices.

A third possibility is to replace \w{\hCasA{\alpha}} by the free simplicial
resolution \w{\Fd(\CasA{\alpha})_{0}} of the $0$-simplices of \w{\CasA{\alpha}}
\wwh that is, of the ordinary (non-simplicial) category underlying \w[.]{\CasA{\alpha}} This version certainly acts simplicially on any realizable mapping algebra, but depends on less data than the other two.
\end{mysubsection}

\begin{remark}\label{rglema}
By definition, we have a weak equivalence of \ww{\SO}-categories \w[,]{\phi:\hCasA{\alpha}\to\CasA{\alpha}} allowing us to pull back any strict extended \Ama \w{\fX:(\CasA{\alpha})\op\to\Sk} to a lax one. If we use the first construction of \S \ref{slat} for \w[,]{\hCasA{\alpha}} the map
$\phi$ is induced by the usual augmentation \w[,]{\var:\Fd\CesA\to\CesA} which sends the free homotopies in \w{\Fd(\CasA{\alpha})_{0}} to identity homotopies.

In particular, for any \w{Y\in\C} we have the \emph{realizable}
lax \Ama \w[,]{\fMAa{\alpha} Y} as well as its
discrete version \w[.]{(\fMAa{\alpha} Y)_{\delta}}

A lax \Ama of the form \w{\fMAa{\alpha} B} for \w{B\in\qF{\alpha}} is called
\emph{free}  (cf.\ \S \ref{dfreema}), and we have as in Lemma \ref{lfreema}:
\end{remark}

\begin{lemma}\label{llfreema}
If $\hfY$ is lax $\alpha$-extended \Ama and \w[,]{B\in\qF{\alpha}}
there is a natural isomorphism
\w[.]{\map_{\MsAa{\alpha}}(\fMAa{\alpha} B,\,\hfY)\cong\hfY\lin{B}}
\end{lemma}

\begin{lemma}\label{llpushout}
Let $B$, $C$, and $D$ be objects in \w{\qF{\alpha}} for some \w[,]{\alpha<\lambda} \w{i:B\hra C} a cofibration and \w{f:B\to D} any map in $\C$, with $E$ the pushout
of \w{D\leftarrow B\hra C} (which thus lies in \w[).]{\qF{\alpha+1}}
Then
\mydiagram[\label{eqpolma}]{
\fMAa{\alpha+1} B \ar[d]^{f_{\ast}} \ar[rr]^{i_{\ast}} && \fMAa{\alpha+1} C \ar[d]\\
\fMAa{\alpha+1} D \ar[rr] && \fMAa{\alpha+1} E
}
\noindent is a homotopy pushout in \w[.]{\MsAa{\alpha+1}}
\end{lemma}

\begin{proof}
Note that since lax \Ama[s] are modeled on \w[,]{\hCasA{\alpha}} rather than \w[,]{\CasA{\alpha}\subseteq\C} the free lax \Ama functor \w{\fMAa{\alpha+1}} does not preserve colimits taken in \w[.]{\CasA{\alpha}} However, since $E$ is a \emph{homotopy} pushout of \w[,]{D\leftarrow B\hra C} by Lemma \ref{lfreema}
\mydiagram[\label{eqpoma}]{
\fMqF{\alpha+1} B \ar[d]^{f_{\ast}} \ar[rr]^{i_{\ast}} && \fMqF{\alpha+1} C \ar[d]\\
\fMqF{\alpha+1} D \ar[rr] && \fMqF{\alpha+1} E
}
\noindent is a homotopy pushout in \w{\MsqF{\alpha+1}} (cf.\ \cite{MathP}).
Because we have a weak equivalence \w{\hCasA{\alpha+1}\to\CasA{\alpha+1}}
(in \w[),]{\SOC} using Lemma \ref{ldema} we see that \wref{llfreema} is a homotopy pushout, too.
\end{proof}

\begin{defn}\label{dllalg}
As in \S \ref{dlalg}, we define an
\ww{\fTAe}-\emph{algebra} to be a discrete $\lambda$-extended \Ama $\fX$
equipped with a natural map \w{\hee{\fX}:\fTAe\fX\to\fX}
such that the diagram corresponding to \wref{eqalgebra} commutes
\emph{up to homotopy} \wh that is:
\begin{myeq}[\label{eqhalgebra}]
\hee{\fX}\circ(\fMAe)_{\delta}\var_{\tLA\rho\fX}~\sim~
\hee{\fX}\circ(\fTAe)_{\delta}(\hee{\fX})~.
\end{myeq}
\end{defn}

\begin{remark}\label{rllalg}
Note that \wref{eqhalgebra} holds automatically for any lax $\lambda$-extended
\Ama $\hfX$, since \w{\hee{\fX}} encodes the (lax) pre-compositions with
maps from \w[,]{\CesA} as in Proposition \ref{ptalg}. Thus lax $\lambda$-extended \Ama[s] are in particular \ww{\fTAe}-algebras. However, 
the analogue of Remark \ref{rptalg} does not hold here, since \wref{eqhalgebra} does not guarantee higher homotopy commutativity.
\end{remark}

%
%
\begin{prop}\label{plema}
Given two weakly equivalent fibrant simplicial sets $X$ and $Y$,
any lax \Ama structure $\hfX$ on $X$ determines a lax \Ama structure
$\hfY$ on $Y$, unique up to weak equivalence.
\end{prop}

\begin{proof}
Let \w{\CX} denote the \ww{\SO}-category of \S \ref{ddema},
where $\OO$ is the object set of \w{\CesA} (since $X$ is fibrant, we
may assume that \w{\CX} is fibrant in \w[).]{\SOC} By Lemma \ref{ldema},
the lax \Ama structure $\hfX$ on $X$ is determined by a map of \ww{\SO}-categories \w[.]{f:\hCesA\to\CX}

Since $X$ and $Y$ are weak equivalent Kan complexes, there is a homotopy equivalence \w{h:Y\to X} by \cite[I, \S 1, Theorem 1]{QuiH}. Factoring $h$ as a cofibration followed by a fibration and using standard lifting properties, we obtain a diagram of weak equivalences:
\mydiagram[\label{eqtriangle}]{
&& Z \ar@{->>}[rrd]_{p} \ar@/_{1.9pc}/[lld]_{q}&&\\
Y \ar@{^{(}->}[rru]_{i} \ar[rrrr]^{h} &&&&
X \ar@/_{1.9pc}/[llu]_{j}
}
\noindent with $p$ a fibration, $i$ a cofibration, \w[,]{p\circ i=h}
\w[,]{q\circ i=\Id_{X}} and \w{p\circ j=\Id_{Y}} (so $j$ is a cofibration, too).

If we use the usual functorial constructions of products, (homotopy) pullbacks and
sequential limits in \w[,]{\Sa} we may assume that they preserve objectwise fibrations and cofibrations, as well as weak equivalences. Thus if \w{L_{I}(X)} is one of the objects of \w[,]{\CX} obtained from $X$ by iterated homotopy limits as above, we have a trivial cofibration \w[,]{L_{I}(j):L_{I}(X)\to L_{I}(Z)}
inducing a commuting diagram of weak equivalences on mapping spaces
\mydiagram[\label{eqsections}]{
\mapa(L_{I}(X),\,X) \ar[rr]^{L_{I}(p)^{\ast}}_{\simeq} \ar[d]_{\Id} &&
\mapa(L_{I}(Z),\,X) \ar[rr]^{j_{\ast}}_{\simeq} &&
\mapa(L_{I}(Z),\,Z) \ar[lld]^{L_{I}(j)^{\ast}}_{\simeq} \\
\mapa(L_{I}(X),\,X) && \mapa(L_{I}(X),\,Z) \ar[ll]^{p{\ast}}_{\simeq}
}
\noindent with \w{L(i)^{\ast}} and \w{p_{\ast}} fibrations
(by \cite[II, \S 2, SM7]{QuiH}). By \wref[,]{eqdema} the trivial fibrations
\w{L(i)^{\ast}\circ p_{\ast}:\mapa(L_{I}(Z),Z)\to\mapa(L_{I}(X),X)} fit together to define a trivial fibration \w{\Phi:\CZ\to\CX} in \w{\SOC} (with section
\w[).]{\Psi:\CX\to\CZ}

Since \w{\hCesA} is cofibrant, we have a lifting in \w[:]{\SOC}
\mydiagram[\label{eqliftalg}]{
\ast \ar[d] \ar[rr] && \CZ \ar@{->>}[d]^{\Phi}_{\simeq} \\
\hCesA \ar[rr]_{f} \ar@{.>}[rru]^{\hat{f}} && \CX
}
\noindent and thus a lax \Ama structure on $Z$ (corresponding to
$\hat{f}$ by Lemma \ref{ldema}).

Similarly, the left hand side of \wref{eqtriangle} yields weak equivalences
\w[,]{L(i)_{\ast}\circ q^{\ast}:\mapa(L_{I}(Z),Z)\to\mapa(L_{I}(Y),Y)} which fit together to define a weak equivalence of \ww{\SO}-categories
\w[.]{\Xi:\CZ\to\CY} Composing this with $\hat{f}$ yields a map
\w[,]{\Xi\circ\hat{f}:\hCesA\to\CY} and thus a lax \Ama structure $\hfY$ on $Y$.
Since $\Xi$ and $\Phi$ are weak equivalences, $\hfY$ is unique up to weak equivalence.
\end{proof}

\begin{mysubsection}{The relaxed colimit construction}
\label{slscc}
We now show how the construction of Section \ref{crma} can be applied
to a lax \Ama structure $\hfX$ on mapping space \w[:]{X=\map(A,Y)}

In the notation of \S \ref{sia}, for any ordinal
\w[,]{\alpha<\lambda} at the $\alpha$-th stage of the
induction we assume that we have constructed the approximations
\w[,]{\qY{\beta}\in\TesA} \wb[,]{\beta\leq\alpha} starting with
\w{\qY{0}:=\ast\in\TesA} as the $0$-th approximation to $Y$.

At stage $\alpha$ of the induction we assume that we have a map of lax
\Ama[s] \w{\qf{\beta}:\fMAe\qY{\beta}\to\hfX} and
a map \w{\tqp{\beta}:\tLA\rho\hfX\to\tqY{\beta}} in $\C$
for each \w[,]{\beta<\alpha} as well as a cofibration
\w{\qi{\beta}:\qY{\beta}\to\qY{\beta+1}} in $\C$ such that
\begin{myeq}\label{eqhexp}
\qf{\beta+1}\circ\fMAe\qi{\beta}~\sim~\qf{\beta}\hs\text{and}\hs
\qf{\beta+1}\circ\fMAe\qp{\beta+1}~\sim~\hee{\hfX}
\end{myeq}
\noindent for each \w{\beta<\alpha} (homotopic, not strictly equal).
By Remark \ref{rllalg}, we have an \ww{\fTAe}-algebra structure
\w{\hee{\hfX}} on \w[.]{\hfX}

The construction of \w{\qY{\alpha+1}} is identical to that of \S \ref{sic}, with
\w{\fMqF{\alpha}} and \w{\fTqF{\alpha}} replaced throughout by \w{\fMAe} and \w[,]{\fTAe} respectively.

Again \w{\qY{\alpha+1}} is defined by the analogue of \wref[:]{eqmapo}
\mydiagram[\label{eqlmapo}]{
\ar @{} [drrr] |>>>>>>>{\framebox{\scriptsize{PO}}}
L_{A}\qY{\alpha}~\ar[d]^{\var_{\qY{\alpha}}}~~ \ar@{=}[r] &
~\tLA\rho\fMAe\qY{\alpha} \ar@{^{(}->}[rr]^-{\wqf{\alpha}} &&
\wLA{\fX} \ar[d]^{\qp{\alpha+1}}\\
\qY{\alpha}~~  \ar@{^{(}->}[rrr]_{\qi{\alpha}} &&& ~~\tqY{\alpha+1}
}
\noindent in $\C$, which is a (homotopy) pushout of objects in \w[,]{\qF{\alpha}} so it is in \w[,]{\qF{\alpha+1}} and thus in \w[,]{\TesA} too.
Applying \w{\fMAe} to \wref{eqlmapo} thus yields a homotopy pushout
diagram in \w[,]{\MsAe} by Lemma \ref{llpushout}.

Moreover, the inner square in the following diagram commutes up to homotopy in \w[:]{\MsAe}
\mydiagram[\label{eqhinmapo}]{
\fMAe L_{A}\qY{\alpha}\ar[d]^{\fMAe\tvar_{\qY{\alpha}}}
\ar[rrr]^-{\fMAe\tLA\rho\qf{\alpha}} &&&
\fMAe\tLA\rho\hfX \ar@/^/[ddr]^{\hee{\hfX}}
\ar[d]_{\fMAe\qp{\alpha+1}}\\
\fMAe\tqY{\alpha}  \ar[rrr]^{\fMAe\qi{\alpha}}
\ar@/_/[drrrr]_{\qf{\alpha}} &&&
\fMAe\qY{\alpha+1} \ar@{.>}[dr]_<<<<<{\qf{\alpha+1}} & \\
& & & & \hso \hfX~.
}
\noindent by the induction hypothesis.

Since the upper left corner of \wref{eqhinmapo} consists of free lax
\Ama[s,] by Lemma \ref{llfreema} we have a dotted map (of lax \Ama[s]) making
the full diagram \wref{eqhinmapo} commute up to homotopy. This completes the induction step by defining \w[.]{\qf{\alpha+1}}

If $\alpha$ is a limit ordinal, we let
\w[,]{\qY{\alpha}:=\colim_{\beta<\alpha}\qY{\beta}} as before.
\end{mysubsection}

The proof of Theorem \ref{treal} readily adapts to yield:

%
%
\begin{thm}\label{threal}
Let $A$ be any $\lambda$-small object in a pointed simplicial model
category $\C$, and $X$ a pointed simplicial set equipped with a lax
$\lambda$-extended \Ama structure $\hfX$.  If we construct
\w[,]{Y:=\tqY{\lambda}} as in \S \ref{slscc}, then the map
of lax \Ama[s] \w{\qf{\lambda}:\fMAe Y\to\hfX} defined there induces a
weak equivalence \w[.]{X\simeq\map_{\C}(A,Y)}
\end{thm}

\begin{remark}\label{rsstover}
The procedure described in Section \ref{crma} for realizing a
\emph{strict} \Ama $\fX$ can be simplified when \w{\C=\Ta} for certain choices of $A$ \wh e.g., when $A$ is a wedge of (possibly localized) spheres. In this case 
we can use the Stover resolution for $\fX$, obtained by iterating the monad \w{\fTAd} and applying \w{\tLA} to the resulting simplicial \Ama \w[,]{\fVd}
with the missing face maps provided by the \ww{\fTAd}-algebra
structure \w{\var_{\fX}:\fTAd\fX\to\fX} (see \cite{BBlaC}). The ingredients allowing this simplified approach are:

\begin{enumerate}
\renewcommand{\labelenumi}{(\alph{enumi})~}
\item The fact that the loop space functor $\Omega$ commutes with the realization of simplicial spaces (which are connected CW complexes in each dimension) \wh see
\cite[Theorem 12.3]{MayG}. 
\item The fact that \w{-\otimes R} is exact when $R$ is a subring of the rationals.
\item The fact that when \w[,]{A=A_{0}\vee A_{1}} an \Ama encodes also the \ww{A_{0}}- and \ww{A_{1}}-\ma structures.
\end{enumerate}

\end{remark}

\begin{mysubsection}{Lax Stover resolutions}
\label{slsr}
The usual Stover construction is not homotopy invariant, but it has a lax version
which can be described as follows: assume that $A$ is
$\lambda$-compact, so that \w{\mapa(A,-)} commutes with sequential
colimits of length $\lambda$, and let \w{\TsA} be the smallest subset
of \w{\Obj\C} containing $A$ and closed under suspensions and
coproducts of cardinality less than $\lambda$. This suffices to define
the comonad \w{L_{A}:=\tLA\circ\rho_{0}\fMAd} on $\C$, as well as
the monad \w{\fTA:=\rho_{0}\fMAd\circ\tLA} on the category \w{\MsAd}
of discrete \Ama[s.]

Let \w{\CsA} be the full sub-simplicial category of $\C$ with objects
\w[,]{\TsA}, and \w{\hCsA=\F\pi_{0}(\CsA)} the free simplicial
resolution of its homotopy category. Given a discrete \Ama $\fX$ with
a lax \Ama structure \w{\hfX:\hCsA\to\Sk} on \w[,]{X=\fX\lin{A}}
we obtain a lax \ww{\fTA}-algebra structure
\w{\var_{\fX}:\fTA\fX\to\fX} on $\fX$, so that we have a homotopy
\begin{myeq}\label{eqhtpyalg}
H:\var_{\fX}\circ\fMAd(\var_{\tLA\rho\fX})~\xra{\sim}~
\var_{\fX}\circ \fTA(\var_{\fX})
\end{myeq}
\noindent as in \wref[.]{eqhalgebra}

If we set \w{W_{n}:=\tLA\fTA^{n}\rho\fX} for each \w[,]{n\geq 0}
the monad \w{\fTA} defines degeneracies and face maps making this into
a simplicial object over $\C$, except for the \emph{last} face map
\w[,]{d_{n}:W_{n}\to W_{n-1}} which we define to be
\w{\tLA\fTA^{n}\var_{\fX}} (for \w[).]{n\geq 0} All the simplicial
identities hold on the nose, except those involving the last face map,
which hold up to homotopy (with the homotopy induced by $H$). Thus for
example \w{d_{0}d_{1}\sim d_{0}d_{0}} is just \w{\tLA} applied to
\wref[.]{eqhtpyalg} However, in order to get higher homotopies, we
need the full lax \Ama structure on $\hfX$, which make
\w{\var_{\fX}:\fTA\fX\to\fX} into an $\infty$-$\fTA$-algebra structure
on $\fX$.  With this extra structure, \w{\Wd} becomes an
$\infty$-homotopy commutative simplicial space, which can then be
rectified to a strict simplicial space \w[,]{\Vd} using
\cite[Theorem~IV.4.37]{BVogHI} or \cite[Theorem~2.4]{DKSmH}
Applying the Bousfield-Friedlander spectral sequence
(see \cite[Theorem~B.5]{BFrieH}), we then deduce that the geometric realization \w{\|\Vd\|} of \w{\Vd} realizes the original \Ama $\fX$.
\end{mysubsection}

%
%
\section*{Appendix: \ A Dold-Lashof type approach}
\label{adla}
\setcounter{thm}{0}
\setcounter{section}{7}

Here we sketch an alternative approach to recovering \w{\CWA Y} from \w{\mapp(A,X)} by successive approximations in the spirit of the Dold-Lashof
``projective spaces'' \w{P^{n}\Omega Y} for loop spaces (see \cite{DLashP,MFucD,StaHH}.
Our construction is based directly on that of Dror-Farjoun for \w{\CWA Y} (in \cite[\S 2B]{DroC}).  The setting is still that of \S \ref{sass}.

\begin{mysubsection}{Mapping spaces and the monad \ww{\TA}}
\label{sms}
For fixed cofibrant \w[,]{A\in\C} we define a functor \w{\MA:\C\to\Sk} by
\w[.]{\MA X:=\mapa(A,X)} This has a left adjoint \w{\FA:\Sa\to\C} defined
\w[,]{\FA K:=A\ltimes K:=A\otimes K/A\otimes\ast} when $K$ is pointed (compare \wref[,]{eqhalfsmash} where the other half-smash was used to define $\otimes$ for \w[).]{\C=\Sa} This is actually an enriched adjunction (cf.\  \cite[\S 6.7]{BorcH2}) \wh i.e., there are natural isomorphisms of (pointed) simplicial sets
$$
\map_{\C}(\FA K, X)=\map_{\C}(A\ltimes K, X)\cong
\map_{\Sa}(K,\map_{\C}(A,X))=\map_{\Sa}(K,\MA X)~.
$$

This defines a monad \w[,]{\TA:=\MA\FA:\Sa\to\Sa} equipped
with natural transformations \w{\eta:\Id\to\TA} (adjoint to
\w[)]{\Id_{\FA}} and \w{\mu:\TA\TA\to\TA} (which is \w{\MA} applied to
the adjoint of \w[),]{\Id_{\TA}} satisfying the usual identities. A
\ww{\TA}-\emph{algebra} is a pointed simplicial set \w{X\in\Sa} equipped with
a splitting \w{\var_{X}:\TA X\to X} for \w[,]{\eta_{X}:X\to\TA X} such that:
\mydiagram[\label{eqalgeb}]{
\TA\TA X \ar[r]^{\mu_{X}} \ar[d]_{\TA(\var_{X})} & \TA X \ar[d]_{\var_{X}} \\
\TA X \ar[r]_{\var_{X}} & X
}
\noindent commutes (cf.\ \cite[\S 4.1]{BorcH2}.

Note that when \w{X=\MA Y} for some \w[,]{Y\in\C} we may take
\w{\var_{\MA Y}:\MA\FA\MA Y\to\MA Y} to be \w[,]{\MA(\widetilde{\Id_{\MA Y}})}
in which  case \wref{eqalgeb} is satisfied automatically. Thus every
$A$-mapping space is in particular an algebra over \w[.]{\TA} We wish
to find conditions for the converse to hold.
\end{mysubsection}

\begin{mysubsection}{Reconstructing the target}\label{srt}
Let $A$ be a fixed object in a pointed simplicial model category
$\C$. For any \w[,]{X=\mapp_{\C}(A,Y)} we describe a (transfinite)
inductive procedure for recovering $Y$ from $X$ (up to
$A$-equivalence), together with its \ww{\TA}-algebra structure. This
procedure is essentially Dror-Farjoun's construction of the
\ww{\CWA{}}-approximation of $Y$ (cf.\ \cite[\S 2B]{DroC}), modified
to take into account the fact that $Y$ is not actually given.

We start with \w{\qZ{0}:=\ast} as the $0$-th approximation to $Y$ in $\C$.
For any ordinal \w[,]{\alpha<\lambda} at the $\alpha$-th stage of the
induction we assume that for each \w[,]{\beta\leq\alpha} we have
constructed \w[,]{\qZ{\beta}\in\C} equipped with maps of
\ww{\TA}-algebras \w{\qf{\beta}:\MA\qZ{\beta}\to X} as well as maps
\w{\qp{\beta}:\FA X\to\qZ{\beta}} in $\C$, together with an injective map of
simplicial sets \w{\qi{\beta}:\qZ{\beta}\to\qZ{\beta+1}} such that
\begin{myeq}\label{eqexpand}
\qf{\beta+1}\circ\MA\qi{\beta}~=~\qf{\beta}\hs\text{and}\hs
\qf{\beta+1}\circ\MA\qp{\beta+1}~=~\var_{X}
\end{myeq}
\noindent for each \w[.]{\beta<\alpha}

If $\alpha$ is a limit ordinal, we let
\w[,]{\qZ{\alpha}:=\colim_{\beta<\alpha}\qZ{\beta}} with
\w[.]{\qf{\alpha}:=\colim_{\beta<\alpha}\qf{\beta}}

We now define \w{\qZ{\alpha+1}} by the homotopy pushout diagram
\mydiagram[\label{eqhpo}]{
\ar @{} [drr] |>>>>>{\framebox{\scriptsize{PO}}}
\FA\MA\qZ{\alpha}  \ar[rr]^-{\FA\qf{\alpha}} \ar[d]^{\var_{\qZ{\alpha}}} &&
\FA X \ar[d]^{\qp{\alpha+1}}\\
\qZ{\alpha}  \ar[rr]_{\qi{\alpha}} && \qZ{\alpha+1}
}
\noindent in $\C$.

Note that \emph{if} \w[,]{X=\MA Y} and we assume by induction that
\w{\qf{\alpha}:\MA\qZ{\alpha}\to\MA Y=X} is induced by
\w{\qe{\alpha}:\qZ{\alpha}\to Y} making the outer square in the
following diagram commute (up to homotopy):
\mydiagram[\label{equsehpo}]{
\FA\MA\qZ{\alpha}  \ar[rr]^-{\FA\qf{\alpha}} \ar[d]^{\var_{\qZ{\alpha}}} &&
\FA X \ar[d]_{\qp{\alpha}} \ar@/^/[ddr]^{\var_{Y}}  & \\
\qZ{\alpha}  \ar[rr]^-{\qi{\alpha}} \ar@/_/[drrr]_{\qe{\alpha}} & &
\qZ{\alpha+1} \ar@{.>}[dr]_<<{\qe{\alpha+1}} & \\
& & & Y
}
\noindent then the pushout property guarantees the existence
of map \w[,]{\qe{\alpha+1}:\qZ{\alpha+1}\to Y} unique up to homotopy,
such that \wref{eqexpand} holds for \w[.]{\qf{\alpha+1}:=\MA\qe{\alpha+1}}

Applying \w{\MA} to \wref[,]{eqhpo} we obtain a commuting square in \w[:]{\Sa}
\mydiagram[\label{eqsquare}]{
\TA\MA\qZ{\alpha}  \ar[rr]^{\TA\qf{\alpha}} \ar[d]^{\MA\var_{\qZ{\alpha}}} & &
\TA X \ar[d]^{\MA\qp{\alpha}}\\
\MA\qZ{\alpha}  \ar[rr]_{\MA\qi{\alpha}} && \MA\qZ{\alpha+1}
}
\noindent in $\C$ \wh though it is no longer homotopy cocartesian.

Nevertheless, if we assume that \w{\qf{\alpha}} is a map of \ww{\TA}-algebras,
the following square commutes:
$$
\xymatrix{
\TA\MA\qZ{\alpha}  \ar[rr]^-{\TA\qf{\alpha}} \ar[d]^{\var_{\MA Y}=
\MA\var_{\qZ{\alpha}}} &&
\TA X \ar[d]^{\var_{X}}\\
\MA\qZ{\alpha}  \ar[rr]_{\qf{\alpha}} & & X
}
$$
\noindent Thus it makes sense to look for a dotted map (of
\ww{\TA}-algebras) making the following diagram commute:
\mydiagram[\label{eqinducehpo}]{
\TA\MA\qZ{\alpha}  \ar[rr]^{\TA\qf{\alpha}}
\ar[d]_{\var_{\MA Y}=\MA\var_{\qZ{\alpha}}} &&
\TA X \ar[d]_{\qp{\alpha}} \ar@/^/[ddr]^{\var_{X}}  & \\
\MA\qZ{\alpha}  \ar[rr]^-{\MA\qi{\alpha}} \ar@/_/[drrr]_{\qf{\alpha}} &&
\MA\qZ{\alpha+1} \ar@{.>}[dr]_<<{\qf{\alpha+1}} & \\
& & & X
}
\end{mysubsection}

\begin{defn}\label{dobstr}
Given a \ww{\TA}-algebra $X$ and a sequence of objects \w{\qZ{\beta}}
equipped with maps \w[,]{\qf{\beta}} \w[,]{\qp{\beta}} and
\w{\qi{\beta}} satisfying \wref{eqexpand} for \w[,]{\beta\leq\alpha}
the \ww{(\alpha+1)}-\emph{st lifting} for this sequence is any map of
\ww{\TA}-algebras \w{\qf{\alpha+1}:\MA\qZ{\alpha+1}\to X} satisfying
\wref{eqexpand} for \w[,]{\alpha+1} where \w[,]{\qZ{\alpha+1}}
\w[,]{\qp{\alpha+1}} and \w{\qi{\alpha}} are defined by \wref[.]{eqhpo}
\end{defn}

%
%
\begin{thm}\label{tobstr}
Let \w{A\in\Sa} be a fixed simplicial set which is $\lambda$-small for \w{\lambda=\lambda(A)} (Definition \ref{dsmall}). If $X$ is a \ww{\TA}-algebra and \w{\qZ{\alpha}\in\C} \wb{\alpha\leq\lambda} is a sequence of objects
in $\C$ as above having $\alpha$-liftings
\w{\qf{\alpha}:\MA\qZ{\alpha}\to X} for each \w[,]{\alpha<\lambda}
then $X$ is realizable as an $A$-mapping space \w{X\simeq\mapp(A,Y)} for
\w[.]{Y:=\hocolim_{\alpha<\lambda}\qZ{\alpha}}
\end{thm}

\begin{proof}
See \cite[\S 2 Proposition B.1]{DroC}, which is based on Bousfield's small object argument (cf.\ \cite{BousL}).
\end{proof}

\end{document}